\documentclass[11pt]{amsart}
\usepackage{pinlabel}
\usepackage{amsmath, amsthm, amssymb, amscd, mathrsfs, eucal}
\usepackage{enumerate}
\usepackage[T1]{fontenc}
\usepackage{hyperref}
\usepackage{url}
\usepackage{float}
\usepackage{mathtools}
\usepackage{color}
\usepackage[margin=1in]{geometry}
\usepackage{graphicx}

\newtheorem{theorem}{Theorem}[section]
\newtheorem{lemma}[theorem]{Lemma}
\newtheorem{corollary}[theorem]{Corollary}
\newtheorem{conjecture}[theorem]{Conjecture}
\newtheorem{proposition}[theorem]{Proposition}
\newtheorem{example}[theorem]{Example}
\newtheorem{question}[theorem]{Question}

\newtheorem*{claim*}{Claim}

\theoremstyle{definition}
\newtheorem{definition}[theorem]{Definition}

\theoremstyle{theorem}
\newtheorem{theoremA}{Theorem}

\theoremstyle{remark}
\newtheorem{remark}[theorem]{Remark}

\usepackage{color}
\usepackage{pdfcolmk}

\newcommand{\rom}[1]{\text{\uppercase\expandafter{\romannumeral #1\relax}}}

\newcommand{\inj}{\hookrightarrow}

\def\Z{\mathbb Z}

\def\R{\mathbb R}

\def\xtt{\mathtt{x}}

\def\Pcal{\mathcal{P}}
\def\Tcal{\mathcal{T}}
\def\defeq{\vcentcolon=}

\def\scl{\mathrm{scl}}

\newcommand{\bgamma}[1]{{\gamma_{#1}^{-1}}}

\makeatletter
\newsavebox{\@brx}
\newcommand{\llangle}[1][]{\savebox{\@brx}{\(\m@th{#1\langle}\)}%
	\mathopen{\copy\@brx\kern-0.5\wd\@brx\usebox{\@brx}}}
\newcommand{\rrangle}[1][]{\savebox{\@brx}{\(\m@th{#1\rangle}\)}%
	\mathclose{\copy\@brx\kern-0.5\wd\@brx\usebox{\@brx}}}
\makeatother

\numberwithin{equation}{section}

\usepackage{pdfpages}

\title{The Kervaire conjecture and the minimal complexity of surfaces}
\author{Lvzhou Chen}
\address{Department of Mathematics\\ Purdue University\\ West Lafayette, Indiana, USA}
\email[L.~Chen]{lvzhou@purdue.edu}

\begin{document}
\maketitle

\begin{abstract}
	The Kervaire conjecture asserts that adding a generator and then a relator to a nontrivial group always results in a nontrivial group. We introduce new methods from stable commutator length to study this type of problems about nontriviality of one-relator quotients. Roughly, we show that surfaces in certain HNN extensions bounding a given word have complexity no less than the complexity of its boundary. 
	A consequence of this is a Freiheitssatz theorem for HNN extensions, which in particular implies and gives a new proof of Klyachko's theorem that confirms the Kervaire conjecture for torsion-free groups.
	As another application, we also generalize the following theorem of Klyachko--Lurye to HNN extensions: For any group $G$ and the quotient $Q$ of $G\star\Z$ by any proper power $w^m$ with $w\in G\star\Z$ projecting to $1\in\Z$, the natural map $G\to Q$ is injective.
\end{abstract}

\section{Introduction}\label{sec: intro}
Many problems in low-dimensional topology have group-theoretic formulations. 
Some of these topological problems remain unsolved due to our lack of understanding of groups obtained from some rather simple operations.
One example is the following basic question about \emph{one-relator products}, which are one-relator quotients of free products.
\begin{question}
	For a free product $H=\star_{\lambda\in\Lambda} G_\lambda$ of nontrivial groups $\{G_\lambda\}_{\lambda\in \Lambda}$ with $|\Lambda|\ge2$, for which $w\in H$ is the quotient $H/\llangle w\rrangle$ nontrivial, where $\llangle w\rrangle$ is the normal subgroup generated by $w$?
\end{question}

For $|\Lambda|\ge3$, it is conjectured that $H/\llangle w\rrangle$ is nontrivial for every $w\in H$; see for instance \cite[Conjecture 9.5]{Gordon:DehnSurg}. This is a generalization\footnote{The knot group is normally generated by a single element (the meridian) via an easy application of van Kampen's theorem using simply connectedness of $S^3$, and thus, so is the fundamental group of a Dehn surgery as it is a quotient of the knot group.} of the unsolved three summand conjecture in $3$-manifold topology, which asserts that no Dehn surgery of $S^3$ along any knot can have three or more summands in the prime decomposition. The case $|\Lambda|\ge3$ is known when all factors are cyclic groups by a theorem of Howie \cite{Howie_cyclic}. On the other hand, if one can find three free factors such that they are finitely generated perfect groups and the conjecture holds for their free product, then it would give a negative answer to the Wiegold question \cite[Question 5.52]{Kourovka}: Is every finitely generated perfect group normally generated by a single element?

The analogous statement fails when $|\Lambda|=2$, for instance $H/\llangle w \rrangle$ is trivial for $w=ab$ when $H=(\Z/m)\star(\Z/n)$ with natural generators $a,b$ and $m,n$ coprime. 
However, one still expects $H/\llangle w \rrangle$ to be nontrivial for all $w$ when the factors are torsion-free:

\begin{conjecture}\label{conj: torsion-free one-relator product}
	If $A$ and $B$ are nontrivial torsion-free groups, then $H/\llangle w\rrangle$ is nontrivial for any $w\in H=A\star B$.
\end{conjecture}

A weaker statement contributed by Freedman appears on Kirby's (1970s) problem list \cite[Problem 66]{Kirby:oldlist}. On the topological side, this is related to the cabling conjecture \cite{AcunaShort} about irreducibility of Dehn surgeries on knots in $S^3$, which implies the three-summand conjecture; see e.g. \cite[Page 2]{Howie_cyclic}. 
Actually one expects the free factor $A$ to naturally inject into the quotient unless $w$ conjugates into $A$, which is known as Levin's conjecture \cite{Levin}. 
This is known under the stronger assumption that $A$ and $B$ are locally indicable\footnote{A group is locally indicable if every finitely generated nontrivial subgroup surjects onto $\Z$.} by Brodski\u{\i} \cite{Brodskii} and independently Howie \cite{Howie_LocIndFrei} and Short \cite{Short}. 


The goal of this paper is to bring in new tools from the seemingly unrelated study of stable commutator length to tackle such problems; see \cite{Cal:sclbook} for a comprehensive reference to this topic.
Roughly speaking, we show that the complexity (measured by the negative Euler characteristic) of certain surface maps to a $K(H,1)$ space for some HNN extension $H$ is no less than the complexity of its boundary (measured by a geometric degree); see Theorems \ref{thmA: less tech main} and \ref{thmA: main} in Section \ref{subsec: precise} for precise statements. This is a generalization of the so-called spectral gap phenomenon in stable commutator length; see Remark \ref{rmk: comparison} for a comparison and see \cite{Chen:sclfpgap,IK,CH:sclgap} for the most related results.

Restricting to planar surfaces and $H=A\star \Z$ (i.e. the free HNN extension of $A$) with $A$ torsion-free, our result implies the Klyachko theorem \cite{Klyachko} and gives a new proof.
\begin{theoremA}[Klyachko, Theorem \ref{thm: Klyachko}]\label{thmA: Klyachko}
	For any torsion-free group $A$, the natural map $A\to H/\llangle w\rrangle$ induced by the inclusion $A\to H=A\star\Z$ is injective for any $w\in H$ with $p(w)=\pm 1$, where $p:H\to \Z$ is the standard projection to the $\Z$ factor.
\end{theoremA}

The interest in Klyachko's theorem and its (new) proofs is twofold. On the one hand, Theorem \ref{thmA: Klyachko} implies the special case of Conjecture \ref{conj: torsion-free one-relator product} where $B=\Z$, since $H/\llangle w\rrangle$ is clearly nontrivial when $p(w)\neq\pm 1$ as it has a map onto $\Z/p(w)\Z$ induced by $H\stackrel{p}{\to}\Z\to\Z/p(w)\Z$.
On the other hand, it is one of the most important progress on the Kervaire--Laudenbach Conjecture \ref{conj: KL} and the Kervaire Conjecture \ref{conj: Kervaire} below. These conjectures originate from Kervaire's classification of high-dimensional knot groups \cite{Kervaire} and remain open in general.
Although Klyachko's theorem has been known for three decades, no significant breakthrough beyond this has been made towards the more general Conjectures \ref{conj: torsion-free one-relator product}, \ref{conj: KL}, and \ref{conj: Kervaire}. It is our hope that new approaches can lead to further progress.
\begin{conjecture}[Kervaire--Laudenbach]\label{conj: KL}
	For any $H=A\star\Z$, the natural map $A\to H/\llangle w\rrangle$ induced by the inclusion $A\to H$ is injective for any $w\in H$ with $p(w)\neq0$.
\end{conjecture}
\begin{conjecture}[Kervaire]\label{conj: Kervaire}
	$H/\llangle w \rrangle$ is nontrivial for all $w\in H=A\star\Z$ if $A$ is nontrivial.
\end{conjecture}

Another influential progress on Conjectures \ref{conj: KL} and \ref{conj: Kervaire} is the theorem of Gerstenhaber--Rothaus \cite{GerstenhaberRothaus}, proving the conjecture for any $A$ finite (and consequently any $A$ residually finite). It also generalizes to hyperlinear groups as observed by Pestov \cite[Corollary 10.4]{Pestov}. A more extensive summary of known results on these conjectures and their generalizations can be found in the survey \cite{surveyEqnoverGrps} from the view of equations over groups.

As an application of our method to nontrivial HNN extensions, we establish the following new Freiheitssatz theorem for HNN extensions over malnormal subgroups. Recall that a subgroup $H\le G$ is malnormal if $gHg^{-1}\cap H=\{id\}$ for all $g\notin H$.
\begin{theoremA}[Theorem \ref{thm: proper power HNN}]\label{thmA: malnormal HNN}
	Let $H=A\star_{C}$ be an HNN extension associated to isomorphic malnormal subgroups $C_1,C_2\le A$. Then for any $w\in H$ with $p(w)=\pm 1$ and not conjugate to $at^{\pm1}$, the natural map $A\to H/\llangle w^m\rrangle$ is injective for any $m\ge2$.
\end{theoremA}

Actually, we can weaken the malnormality assumption in Theorem \ref{thmA: malnormal HNN} to  the assumption that $aC_i a^{-1}\cap C_i=\{id\}$ for any letter $a$ appearing in a cyclically reduced expression of $w$, for $i=1,2$. A similar strengthening holds for Theorem \ref{thmA: main} below. See Remark \ref{rmk: strengthening} for more details.

Applying Theorem \ref{thmA: malnormal HNN} to the free HNN extension $H=A\star \Z$, this recovers the following theorem of Klyachko--Lurye \cite{KlyachkoLurye}, which works for an \emph{arbitrary} free factor $A$ but assumes the relator to be a proper power compared to Theorem \ref{thmA: Klyachko}. Note that very few partial results about Conjecture \ref{conj: KL} works for an arbitrary free factor $A$.
\begin{theoremA}[Klyachko--Lurye, Theorem \ref{thm: proper power}]\label{thmA: proper power}
	For any group $A$, the natural map $A\to H/\llangle w^m\rrangle$ induced by the inclusion $A\to H=A\star\Z$ is injective for any $w\in H$ with $p(w)=1$ and $m\ge2$.
\end{theoremA}

This supports the following Conjecture \ref{conj: Freiheitssatz for proper power} about one-relator quotients by proper powers.
\begin{conjecture}[Howie {\cite[Problem 6.2]{Howie_generalize}}]\label{conj: Freiheitssatz for proper power}
	For an arbitrary free product, the natural map $A_\lambda \to (\star_\lambda A_\lambda)/\llangle w^m \rrangle$ induced by the inclusion $A_\lambda\to \star_\lambda A_\lambda$ is injective whenever $m\ge2$ and $w$ is not conjugate to an element of $A_\lambda$.
\end{conjecture}

Klyachko--Lurye \cite{KlyachkoLurye} also proved that $(A\star\Z)/\llangle w^m\rrangle$ is hyperbolic relative to the subgroup $A$ if we further assume $m\ge3$ or $A$ has no $2$-torsion in Theorem \ref{thmA: proper power}. Our method implies a linear isoperimetric inequality recovering this result; see Theorem \ref{thm: linear isoperimetric}.




\subsection{More detailed statements of results}\label{subsec: precise}
Now we give the more precise statements on our results about the complexity of surfaces and their relation to the problems above. 

Fix a $K(H,1)$ space $X$ for an HNN extension $H=A\star_{C}$ and an element $w\in H$ not conjugate into $A$. The objects of our study are \emph{$w$-admissible surfaces}, each of which is a continuous map $f:S\to X$ from a compact oriented surface $S$ so that the image of each boundary component $B_i$ of $S$ represents either the conjugacy class of $w^{n_i}$ for some $n_i\neq0\in\Z$ or some conjugacy class in $A$ (for which we take $n_i=0$).
The (geometric) \emph{degree} of $S$, denoted $\deg(S)$, is defined as the sum of $|n_i|$ over all boundary components $B_i$ of $S$. See Definition \ref{def: admissible} for a more general definition.

We actually focus on \emph{boundary-incompressible} $w$-admissible surfaces $S$, which essentially means that boundary components of $S$ representing $w^m$ and $w^{-n}$ cannot cancel in a naive way for $m,n\in\Z_+$; see Definition \ref{def: bdry incompressible} for the precise definition. Any $w$-admissible surface can be simplified to a boundary-incompressible one.

A less technical version of our main result is:
\begin{theoremA}[Corollary \ref{cor: gap for groups without small torsion}]\label{thmA: less tech main}
	For the free HNN extension $H=A\star\Z$ of $A$, where each nontrivial element of $A$ has order at least $n$ for some $2\le n\le\infty$ (which is automatic if $n=2$), for any $w\in H$ with $p(w)=1$, every boundary-incompressible $w$-admissible surface $S$ has
	$$-\chi(S)\ge (1-\frac{1}{n})\deg(S).$$
\end{theoremA}

This implies Theorem \ref{thmA: Klyachko} (and similarly Theorem \ref{thmA: proper power}) for the following reason. If some $a\neq id \in A$ becomes trivial in the quotient $H/\llangle w\rrangle$, then $a$ can be written as a product of conjugates of $w$ and $w^{-1}$ in $H$ (with $h_i\in H$):
$$a=(h_1 w^{\pm 1} h_1^{-1})\cdots (h_k w^{\pm 1} h_k^{-1}).$$
Topologically this is equivalent to a $w$-admissible surface $S$ with $\deg(S)=k$, which is a sphere with $k+1$ disks removed and hence has $-\chi(S)=k-1<\deg(S)$. The inequality in Theorem \ref{thmA: less tech main} (with $n=\infty$ as $A$ is torsion-free) rules out the existence of such surfaces provided that we simplify the equation above to ensure boundary-incompressibility of the surface $S$.

The more general main result works for HNN extensions of a group $A$ over isomorphic subgroups $C_1,C_2$, assuming that the group-subgroup pair $(A,C_i)$ satisfies a \emph{length-$n$ relatively free ($n$-RF)} condition, which essentially assumes that there is no short relation (quantified by $n$) between any $a\in A\setminus C_i$ and $C_i$, $i=1,2$; see Definition \ref{def: n-RF} for the precise definition and Section \ref{subsec: n-RF} for a more detailed discussion on this condition. 
\begin{theoremA}[Theorem \ref{thm: HNN general word}]\label{thmA: main}
	Let $H=A\star_{C}$ be the HNN extension associated to inclusions $C\to A$ with images $C_1,C_2\le A$ such that $(A,C_i)$ is $n$-RF for some $2\le n\le\infty$ and $i=1,2$. Let $p:H\to \Z$ be the projection to $\Z$ that restricts trivially to $A$. Then for any $w\in H$ with $p(w)=\pm 1$ and not conjugate to $at^{\pm1}$ for any $a\in A$, every boundary-incompressible $w$-admissible surface has
	$$-\chi(S)\ge (1-\frac{1}{n})\deg(S).$$
\end{theoremA}

Similar to the applications of Theorem \ref{thmA: less tech main}, this implies Freiheitssatz theorems for HNN extensions (e.g. surface groups as HNN extensions over $\Z$) more general than Theorem \ref{thmA: Klyachko}; see Theorem \ref{thm: HNN Freiheitssatz}. Theorem \ref{thmA: malnormal HNN} also follows since the $2$-RF condition is equivalent to malnormality  (Lemma \ref{lemma: 2-RF}).

Theorem \ref{thmA: main} is analogous to the spectral gap theorem proved by the author and Nicolaus Heuer in the context of (relative) stable commutator length (scl) in graphs of groups \cite[Theorem A]{CH:sclgap}. Our $w$-admissible surfaces are analogous to the admissible surfaces in graphs of groups relative to the vertex groups in the scl sense \cite[Definition 2.12]{CH:sclgap}. The key difference is that here we use the \emph{geometric} degree instead of the algebraic degree in the scl context, and this crucial difference makes the problem harder in our context.

One of the key tool used to prove spectral gap properties of scl is to construct suitable quasimorphisms and apply Bavard's duality \cite{Bavard_duality}. Due to the key difference in the notion of degree, it is unclear if this approach is still applicable here.

However, we are still able to adapt in our context the LP-duality method that the author developed to prove sharp lower bounds and spectral gaps of scl; see \cite[Section 6.3]{Chen:sclBS} and \cite[Section 3.2]{CH:sclgap} for an introduction of this method for scl. We focus on the adaptation of this method to $w$-admissible surfaces in HNN extensions, but the same method applies to graphs of groups and in particular free products of groups, which we leave to future work.


The connection between the Kervaire conjecture and scl was not completely clueless.
Ivanov and Klyachko proved spectral gap results in scl (independent to the author's proof \cite{Chen:sclfpgap}) using the car motion method \cite{IK}, which is the method Klyachko originally used to prove Theorem \ref{thmA: Klyachko}. It was suggested to the author by Danny Calegari back then that the connection might go both ways: Some techniques in scl might also apply to the Kervaire conjecture, which we now confirm.

\subsection{Organization of the paper}\label{subsec: Organization}
This paper is organized as follows: We give basic definitions about $w$-admissible surfaces in Section \ref{sec: adm}, and we introduce a normal form of such surfaces in Section \ref{sec: normal form}. Then we explain the (adapted) LP-duality method in Section \ref{sec: duality} and apply it to prove a main technical result (Theorem \ref{thm: HNN main}) in Section \ref{sec: lower bound}. A brief discussion on the main $n$-RF assumption can be found in Section \ref{subsec: n-RF}. Finally in Section \ref{sec: app} we apply Theorem \ref{thm: HNN main} to prove Theorem \ref{thmA: main} (Theorem \ref{thm: HNN general word}), from which we deduce Theorem \ref{thmA: less tech main} (Corollary \ref{cor: gap for groups without small torsion}) as its special case as well as Theorems \ref{thmA: Klyachko}, \ref{thmA: malnormal HNN} and \ref{thmA: proper power} (Theorems \ref{thm: Klyachko}, \ref{thm: proper power HNN} and \ref{thm: proper power}). Some unsolved related questions are listed in Section \ref{sec: questions}.

\subsection*{Acknowledgment}
The author deeply thanks Danny Calegari for suggesting the potential connection between stable commutator length and the Kervaire conjecture back in 2017. The author is very grateful to Cameron Gordon for re-stimulating the author's interest in such problems and for numerous discussions on this topic, from which a gap in an earlier proof was filled. The author is also grateful to Anton Klyachko for helpful discussions and pointing out that Theorem \ref{thmA: proper power} is known. The author also thanks Daniel Allcock, Jeff Danciger, Francesco Fournier-Facio, Yash Lodha, John Luecke, Jason Manning, Bestvina Mladen, Yi Ni and Henry Wilton for helpful conversations and suggestions. Finally, the author thanks the anonymous referees for numerous suggestions that helped improving the quality of this paper, especially for the strengthening of Theorem \ref{thmA: malnormal HNN} (Theorem \ref{thm: proper power HNN}) as discussed in Remark \ref{rmk: strengthening}.


\section{Admissible surfaces}\label{sec: adm}
Fix a group $H$ with a proper subgroup $A$, and let $X$ be a connected topological space with $\pi_1(X)=H$. In this section, we introduce \emph{$w$-admissible surfaces} associated to an element $w\in H$ not conjugate into $A$. We are mostly interested in the case of an HNN extension $H=A\star_C$, but the definitions make sense in general.

\begin{definition}[$w$-admissible]\label{def: admissible}
	A map $f:S\to X$ from a compact oriented surface $S$ is \emph{$w$-admissible} if the image under $f$ of each component of $\partial S$ either
	\begin{enumerate}
		\item represents a conjugacy class in $A$, or
		\item represents the conjugacy class of $w^n$ for some $n\in\Z\setminus\{0\}$.
	\end{enumerate}
	We refer to the union of boundary components of the first type as the \emph{$A$-boundary} of $S$, and refer to the union of the second type as the \emph{$w$-boundary} of $S$; see Figure \ref{fig: admsurf}. 
	A $w$-boundary is \emph{positive} (resp. \emph{negative}) if the exponent $n$ is positive (resp. negative).
	We allow the $A$-boundary to be empty but make the convention throughout this paper that the $w$-boundary is nonempty for each component of $S$. In particular, each component of $S$ has non-positive Euler characteristic as $w$ is nontrivial.
	
	Although the map $f$ is part of the data, we often abbreviate and refer to $S$ as a $w$-admissible surface by thinking of it as a (singular) subsurface in $X$. When $w$ is understood, we simply call $S$ an admissible surface.
	
	The \emph{degree} of a $w$-boundary component representing $w^n$ is $|n|$. Define the \emph{degree} $\deg(S)$ of a $w$-admissible surface $S$ to be the sum of degrees of all $w$-boundary components. By our convention we have $\deg(S)\in\Z_+$. This is well defined  if no $w^n$ is conjugate to $w^m$ whenever $m\neq n$. Similarly, we define $\deg_+(S)$ (resp. $\deg_-(S)$) to be the sum of degrees only over $w$-boundary components representing $w^n$ for some $n>0$ (resp. $n<0$). We have $\deg(S)=\deg_+(S)+\deg_-(S)$.
\end{definition}

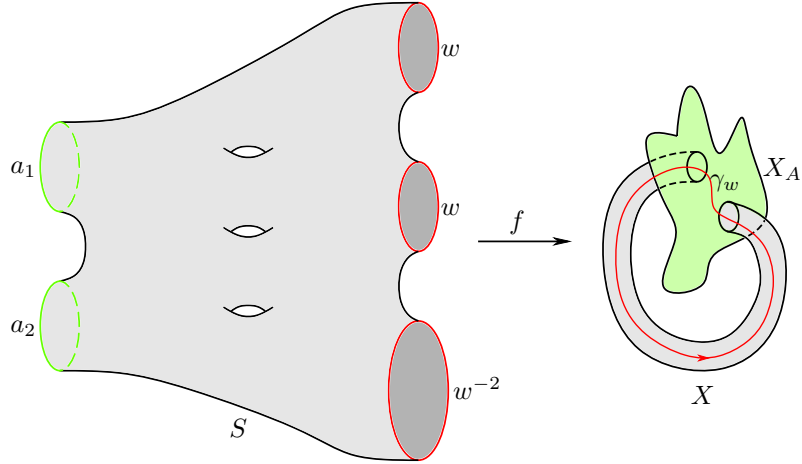
\begin{figure}
	\labellist
	\small \hair 2pt
	\pinlabel $a_1$ at -12 220
	\pinlabel $a_2$ at -12 100
	\pinlabel $S$ at 150 25
	\pinlabel $w^{-2}$ at 330 55
	\pinlabel $w$ at 310 190
	\pinlabel $w$ at 310 310
	\pinlabel $f$ at 360 180
	\pinlabel $X$ at 500 50
	\pinlabel $X_A$ at 560 220
	\pinlabel $\gamma_w$ at 517 210
	\endlabellist
	\centering
	\includegraphics[scale=0.5]{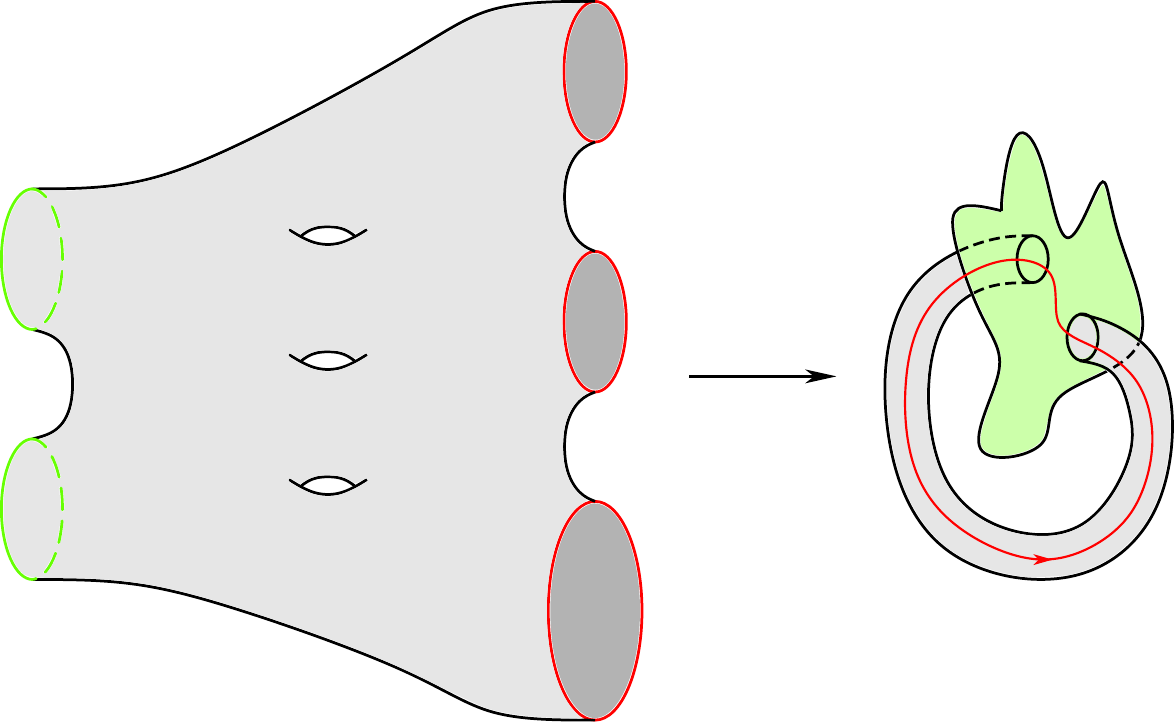}
	\caption{$X$ has a subspace $X_A$ representing the subgroup $A\le H=\pi_1(X)$ and a loop $\gamma_w$ representing some $w\in H$. $S$ is a connected $w$-admissible surface of degree $4$, where the two boundary components of $S$ on the left are $A$-boundary, mapped to conjugacy classes of $a_1,a_2\in A$, and the three on the right are $w$-boundary of $S$, representing powers of $w$.}\label{fig: admsurf}
\end{figure}

\begin{remark}\label{rmk: def}
	One should really refer to these surfaces as $w$-admissible surfaces in $X$ (or $H$) \emph{relative to $A$}. However, in most of this paper, we fix $A$ and $H$ as in an HNN extension $H=A\star_{C}$. The only exception is in the proof of Theorem \ref{thm: HNN general word}, where we pass to a different HNN extension structure on $H$ that enlarges the subgroup $A$.
	Note that a $w$-admissible surface relative to $A$ is also a $w$-admissible surface relative to $A'$ if $A\le A'$ (and $w$ is not conjugate into $A'$).
\end{remark}

There is a similar notion of admissible surfaces in the topological definition of stable commutator length \cite[Notation 2.5]{Cal:rational}; see also \cite[Definition 2.8]{Chen:sclBS}. However, in that context a boundary component representing $w^n$ for $n<0$ is defined to have negative degree or is simply disallowed by considering the so-called monotone admissible surfaces \cite[Lemma 2.7]{Cal:rational}. In our setting we consider the geometric degree instead of the algebraic degree aiming for the Kervaire conjecture.

For an HNN extension $H=A\star_{C}$, there is a surjective homomorphism $p: H\to \Z$ that vanishes on $A$, taking the standard new generator $t$ to a generator of $\Z$ (see the presentation in (\ref{eqn: std presentation}) below).
\begin{lemma}\label{lemma: half degree}
    If $H=A\star_{C}$ is an HNN extension and $p(w)\neq0$, then for any $w$-admissible surface $S$, we have
    $$\deg_+(S)=\deg_-(S)=\frac{1}{2}\deg(S).$$
\end{lemma}
\begin{proof}
    The homomorphism $p$ factors through the abelianization $H_1(H;\Z)=H_1(X;\Z)$. Note that $[\partial S]$ is a trivial first homology class, and $p$ vanishes on $A$, so we must have 
    $$\deg_+(S)\cdot p(w)+\deg_-(S)\cdot p(w^{-1})=0.$$
    As $p(w^{-1})=-p(w)\neq0$, it follows that $\deg_+(S)=\deg_-(S)$, which is half of the total degree $\deg(S)$.
\end{proof}

\begin{definition}[Boundary incompressibility]\label{def: bdry incompressible}
	A $w$-admissible surface $S$ is \emph{boundary-compressible} if there is a compact subsurface $\Sigma\subset S$ which is a pair of pants so that two components of $\partial \Sigma$ are $w$-boundary components of $S$ representing the conjugacy classes $w^n$ and $w^{-m}$ for some $m,n\in\Z_+$ and the third component of $\partial \Sigma$ is a loop in $S$ (with the orientation induced from $\Sigma$) representing the conjugacy class of $w^{m-n}$ (under the map $f:S\to X$); see Figure \ref{fig: bdrycompressible}.
	
	A $w$-admissible surface $S$ is \emph{boundary-incompressible} if it is not boundary-compressible. In particular, given any base point $p$ in a boundary-incompressible surface, for any two $w$-boundary components, their images in $\pi_1(X,p)$ cannot be expressed as $hw^nh^{-1}$ and $hw^{-m}h^{-1}$ for some $h\in \pi_1(X,p)$ and $m,n\in\Z_+$.
\end{definition}

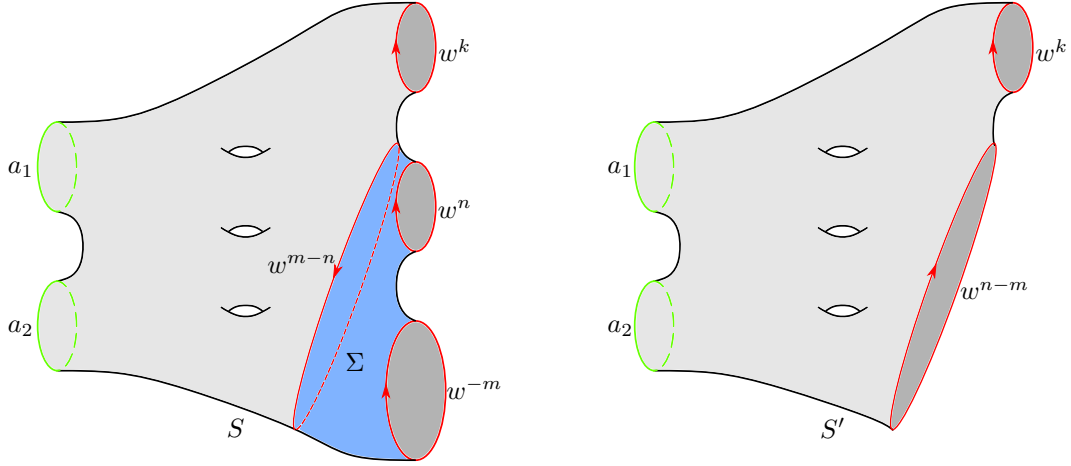
\begin{figure}
	\labellist
	\small \hair 2pt
	\pinlabel $a_1$ at -12 220
	\pinlabel $a_2$ at -12 100
	\pinlabel $S$ at 150 25
	\pinlabel $\Sigma$ at 240 75
	\pinlabel $w^{-m}$ at 330 55
	\pinlabel $w^{m-n}$ at 200 150
	\pinlabel $w^n$ at 315 190
	\pinlabel $w^k$ at 315 310
	
	\pinlabel $a_1$ at 440 220
	\pinlabel $a_2$ at 440 100
	\pinlabel $S'$ at 600 25
	\pinlabel $w^{n-m}$ at 722 130
	\pinlabel $w^k$ at 765 310
	\endlabellist
	\centering
	\includegraphics[scale=0.5]{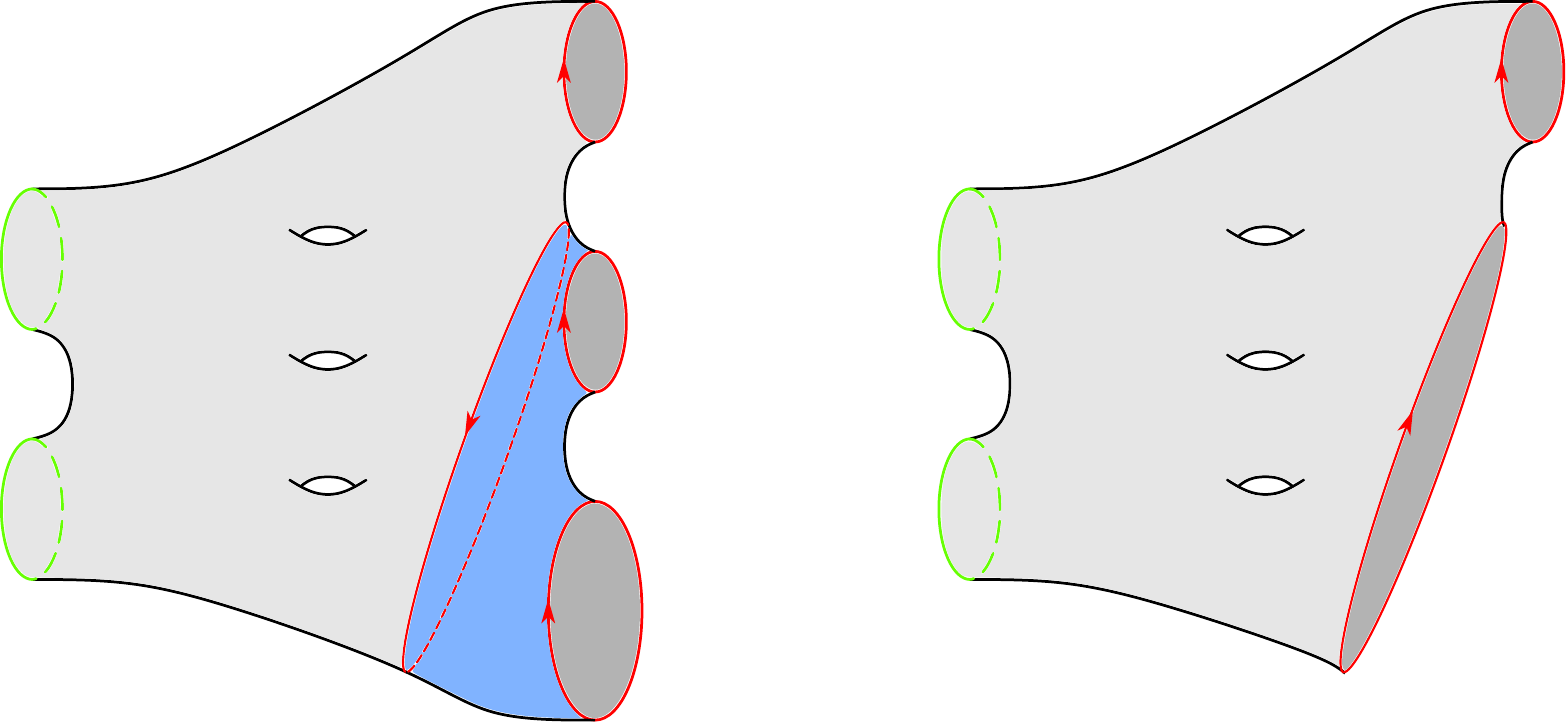}
	\caption{A boundary-compressible $w$-admissible surface $S$ with $k,m,n\in\Z_+$ and the simplified $w$-admissible surface $S'=S\setminus \Sigma$, whose boundary representing $w^{n-m}$ (with the orientation induced from $S'$) needs to be further capped off by a disk if $m=n$}\label{fig: bdrycompressible}
\end{figure}

One can keep simplifying a boundary-compressible $w$-admissible surface until it either has no $w$-boundary left or becomes boundary-incompressible. Indeed, for a pair of pants $\Sigma\subset S$ as in the definition, consider a new surface $S'=S\setminus \Sigma$, where we further cap off the new boundary representing $w^{n-m}$ if $n=m$. The new surface $S'$ has $-\chi(S')\le -\chi(S)-1$ and $\deg(S')=\deg(S)-2\min(m,n)\le\deg(S)-2$; see Figure \ref{fig: bdrycompressible}.


The following example shows how $w$-admissible surfaces naturally correspond to certain kinds of equations in the group $H$. Such equations arise naturally in our application to the Kervaire--Laudenbach conjecture.

\begin{example}\label{example: adm surf from relations}
	Suppose $a\in A\le H$ lies in the normal closure $\llangle w\rrangle$ of $w$, i.e. there is an equation in $H$ of the form 
	\begin{equation}\label{eqn: relation}
	    a=(h_1 w^{n_1} h_1^{-1})\cdots (h_k w^{n_k} h_k^{-1})
	\end{equation}
	for some $k\ge 1$, $n_i\in\Z\setminus\{0\}$, and $h_i\in H$.
	Such an expression provides a $w$-admissible $S$, which is a sphere with $k+1$ disks removed, where one boundary component represents $a$ and the other $k$ components represent the conjugacy classes of $w^{n_i}$ for $i=1,\dots, k$; see Figure \ref{fig: fromeqn}.
	
	If $S$ is boundary-compressible, then the simplified surface $S'$ above gives another expression of $a$ of the form (\ref{eqn: relation}) with a smaller $k$. Hence if $k$ is minimal among all such expressions of $g$, then the corresponding $w$-admissible surface $S$ is boundary-incompressible. 
\end{example}

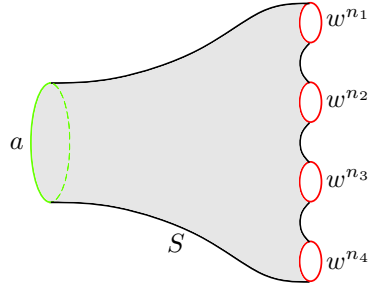
\begin{figure}
	\labellist
	\small \hair 2pt
	\pinlabel $a$ at -10 105
	\pinlabel $S$ at 110 30
	\pinlabel $w^{n_1}$ at 240 200
	\pinlabel $w^{n_2}$ at 240 139
	\pinlabel $w^{n_3}$ at 240 79
	\pinlabel $w^{n_4}$ at 240 17
	\endlabellist
	\centering
	\includegraphics[scale=0.5]{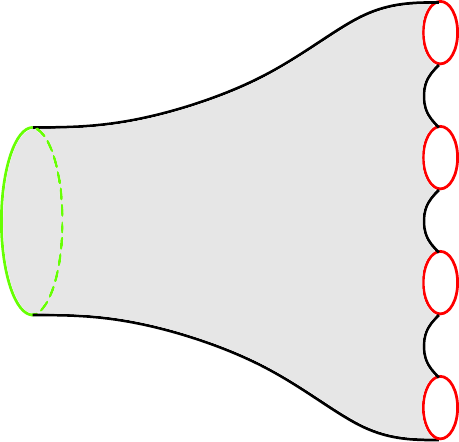}
	\caption{A $w$-admissible surface $S$ corresponding to an equation of the form (\ref{eqn: relation}) in the case $k=4$.}\label{fig: fromeqn}
\end{figure}

Note that the surface in the example above is connected, planar, and has exactly one $A$-boundary component. None of these properties are required for a $w$-admissible surface in general.

\section{A normal form}\label{sec: normal form}
Starting from this section, we focus on an HNN extension $H=A\star_C$ given by two injections $i_P, i_N:C\inj A$. Denote the two images as $C_P$ and $C_N$ respectively.

We develop a normal form for $w$-admissible surfaces, which is a decomposition into disks and annuli with combinatorial boundary information. This is parallel to the normal form in \cite{Chen:sclBS} for admissible surfaces (in the context of stable commutator length) in a graph of spaces. The fact that here we have boundary components representing $w^n$ with $n<0$ does not affect the process of simplifying an admissible surface to put it in the normal form. We include some details for completeness. The discussion below works for any graph of groups, but we focus on the case of an HNN extension for concreteness.

\subsection{Basic setup}\label{subsec: setup}
Let $(X_A,b_A)$ and $(X_C,b_C)$ be based $K(A,1)$ and $K(C,1)$ spaces respectively. The two inclusions of $C$ into $A$ are represented by continuous maps $i_P,i_N: (X_C,b_C)\to (X_A,b_A)$ respectively. Thus we can build the space $X$ as a graph of space, where the graph is just a loop with one vertex, $X_A$ is the vertex space, and $X_C$ is the edge space. Explicitly, $X$ is a quotient of $X_A\sqcup(X_C\times [-1,1])$, where any $(x,-1)\in X_C\times\{-1\}$ is glued to $i_N(x)\in X_A$ and $(x,1)\in X_C\times\{1\}$ is glued to $i_P(x)\in X_A$. The space $X$ built this way is a $K(H,1)$ for $H=A\star_C$.


Note that $X_A$ is naturally a subspace of $X$. We also identify $X_C$ with the subspace $X_C\times\{0\}$ of $X$, which has a product neighborhood $X_C\times(-1,1)$. Cutting $X$ along $X_C$ (and taking completion) yields a space $V$, which is the mapping cylinders associated to $i_P$ and $i_N$ with $X_A$ identified. We call $V$ the \emph{thickened vertex space} and note that it deformation retracts to $X_A$; see the top-right of Figure \ref{fig: graphofspace}. The image of $\{b_C\}\times[-1,1]$ in $X$ is a loop with the standard orientation on $[-1,1]$, which we denote by $\tau$. Denote by $t$ the corresponding element in $\pi_1(X,b_A)=H=A\star_C$. This way we obtain the standard (relative) presentation 
\begin{equation}\label{eqn: std presentation}
    H=\langle A, t\mid i_N(c)=ti_P(c)t^{-1}\text{ for all } c\in C\rangle.
\end{equation}

Fix any $w\in H$ that does not conjugate into $A$. Represent it as a loop $\gamma: S^1\to X$. 
We choose in below a good representative of $\gamma$ in its free homotopy class corresponding to a cyclically reduced expression $w=a_1 t^{e_1}\dots a_k t^{e_k}$ with $e_i=\pm 1$ and $a_i\in A$. We denote $|w|\defeq k$.
For each $a_i\in A$, represent it as a loop $\alpha_i$ in $X_A$ based at $b_A$. In the cyclically reduced expression above, replace each $a_i$ by $\alpha_i$, interpret each $t^{e_i}$ as the loop $\tau$ with the appropriate orientation depending on $e_i$, and replace group operation by concatenation to obtain our representative of $\gamma$. By construction, $\gamma$ intersects $X_C$ transversely exactly $k$ times.
Since $w$ does not conjugate into $A$, we have $k\ge1$, and the intersections with $X_C$ divide $\gamma$ into $k$ segments $\gamma_1,\dots,\gamma_k$, where each $\gamma_i$ starts (resp. ends) with the second (resp. first) half of $t^{e_{i-1}}$ (resp. $t^{e_i}$) and follows $\alpha_i$ in the middle, indices taken mod $k$.
We say such a representative $\gamma$ is \emph{tight} and fix it in the discussion below.

Using the product neighborhood of the edge space $X_C$ homeomorphic to $X_C\times(-1,1)$, each segment $\gamma_i$ starts from (resp. ends at) either the negative or positive side of $X_C$. This divides the above segments into four types, $PP$, $PN$, $NP$, $NN$, where the first (resp. second) letter indicates which side the arc starts from (resp. ends at), with $P$ and $N$ standing for positive and negative respectively. Algebraically, the first (resp. second) letter for the type of $\gamma_i$ is $P$ if $e_{i-1}=1$ (resp. $e_i=-1$). Similarly, this also defines the type of each segment $\bgamma{i}$ so that its type is the type of $\gamma_i$ with the two letters swapped. We will use this in Section \ref{subsec: polygonal boundary}.

\subsection{Putting $S$ in (simple) 
normal form}
Fix any $w$-admissible surface $f:S\to X$. Up to homotopy, we assume that $f$ restricted to each $w$-boundary is a covering map to $\gamma$ (i.e. it factors as $S^1\stackrel{p}{\to} S^1\stackrel{\gamma}{\to} X$ for a covering map $p$ of positive degree), and each $A$-boundary has image in $X_A$. Putting $f$ in general position so that it is transverse to $X_C$, then $f^{-1}(X_C)$ is an embedded proper submanifold of codimension $1$, i.e. a finite disjoint union of embedded loops and proper embedded arcs with endpoints on $w$-boundary components. $f^{-1}(X_C)$ divides each $w$-boundary of degree $n$ into exactly $|n|k$ segments.

The lemma below shows that one can always  homotope $f$ and possibly simplify $S$ so that $f^{-1}(X_C)$ has no embedded loops.

\begin{lemma}\label{lemma: no loop}
	For each $w$-admissible surface $f:S\to X$, there is another $w$-admissible surface $g:S'\to X$ with $\deg(S')=\deg(S)$ and $-\chi(S')\le-\chi(S)$ such that $f^{-1}(X_C)$ has no embedded loops.
\end{lemma}
\begin{proof}
	By transversality, $f^{-1}(X_C)$ is a disjoint union of finitely many embedded loops in $S$. 	
	To prove the lemma, we modify $S$ and $f$ to reduce the number of loops (i.e. components) in $f^{-1}(X_C)$.
	
	For any embedded loop $L$ in $f^{-1}(X_C)$, if its image is null homotopic in $X$ then the restriction of $f$ to $L$ extends to a disk $D$. Moreover, since $X_C$ is $\pi_1$-injective, we may assume that $f(D)\subset X_C$. In this case we can compress $S$ along $L$ (i.e. cut $S$ along $L$ and glue with two copies of $D$) to obtain $S'$ and a map $g:S'\to X$ using the extension of $f|_L$ on $D$ above. Homotope $g$ to push $g(D)$ in the direction away from $X_C$ so that $g^{-1}(X_C)$ has one less component (corresponding to $L$) than $f^{-1}(X_C)$. $S'$ has all the required properties and $\chi(S')=\chi(S)+2$, except that $S'$ might have a component $\Sigma$ that has no $w$-boundary. To make sure that $S'$ meets our convention, in this situation we simply remove this component from $S'$ and $\chi(S'\setminus \Sigma)=\chi(S')-\chi(\Sigma)\ge \chi(S)$ as desired.
	
	If the image of $L$ represents a nontrivial conjugacy class, then we cut $S$ along $L$ to obtain a new surface $S'$ with $\chi(S')=\chi(S)$ and a map $g: S'\to X$ induced by $f$. Pushing the two new boundary components corresponding to $L$ away from $X_C$ and into $X_A\subset X$, this reduces the number of embedded loops in the preimage of $X_C$ and makes $S'$ into a $w$-admissible surface, which simply has two more $A$-boundary components compared to $S$. In particular $\deg(S')=\deg(S)$. In the special case where $L$ cuts out a component $\Sigma$ that has no $w$-boundary, since $f(L)$ represents a nontrivial class in $X$, we see that $\chi(\Sigma)\le 0$ and hence removing it from $S'$ gives the desired inequality $-\chi(S')\le-\chi(S)$.
	
	Repeating the two procedures above on each embedded loop in $ f^{-1}(X_C)$ completes the proof.
\end{proof}

Now suppose $F\defeq f^{-1}(X_C)$ is a finite disjoint union of embedded proper arcs with endpoints on $w$-boundary components; see Figure \ref{fig: graphofspace}.
It follows that $S\setminus F$ has two types of boundary components:
\begin{enumerate}
	\item $A$-boundary components, exactly corresponding to those on $S$;
	\item \emph{polygonal boundary components}, each of which is divided into an even number of sides that alternate between arcs in $F$ and segments on some $w$-boundary of $S$. Its structure will be discussed in more detail in Section \ref{subsec: polygonal boundary}.
\end{enumerate}

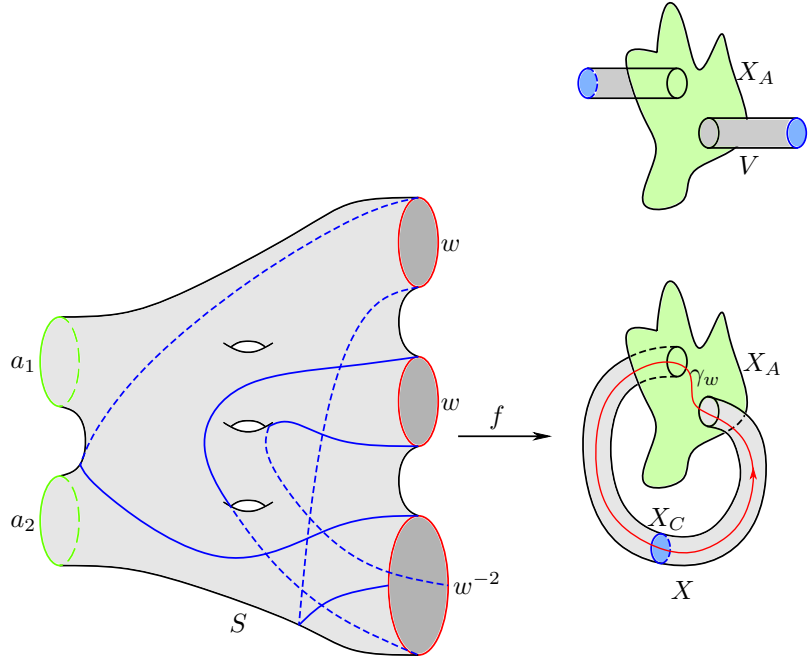
\begin{figure}
	\labellist
	\small \hair 2pt
	\pinlabel $a_1$ at -12 220
	\pinlabel $a_2$ at -12 100
	\pinlabel $S$ at 150 25
	\pinlabel $w^{-2}$ at 330 55
	\pinlabel $w$ at 310 190
	\pinlabel $w$ at 310 310
	
	\pinlabel $f$ at 345 180
	
	\pinlabel $X_C$ at 473 103
	\pinlabel $X$ at 485 50
	\pinlabel $X_A$ at 545 220
	\pinlabel $\gamma_w$ at 502 210
	
	\pinlabel $V$ at 535 370
	\pinlabel $X_A$ at 540 440
	\endlabellist
	\centering
	\includegraphics[scale=0.5]{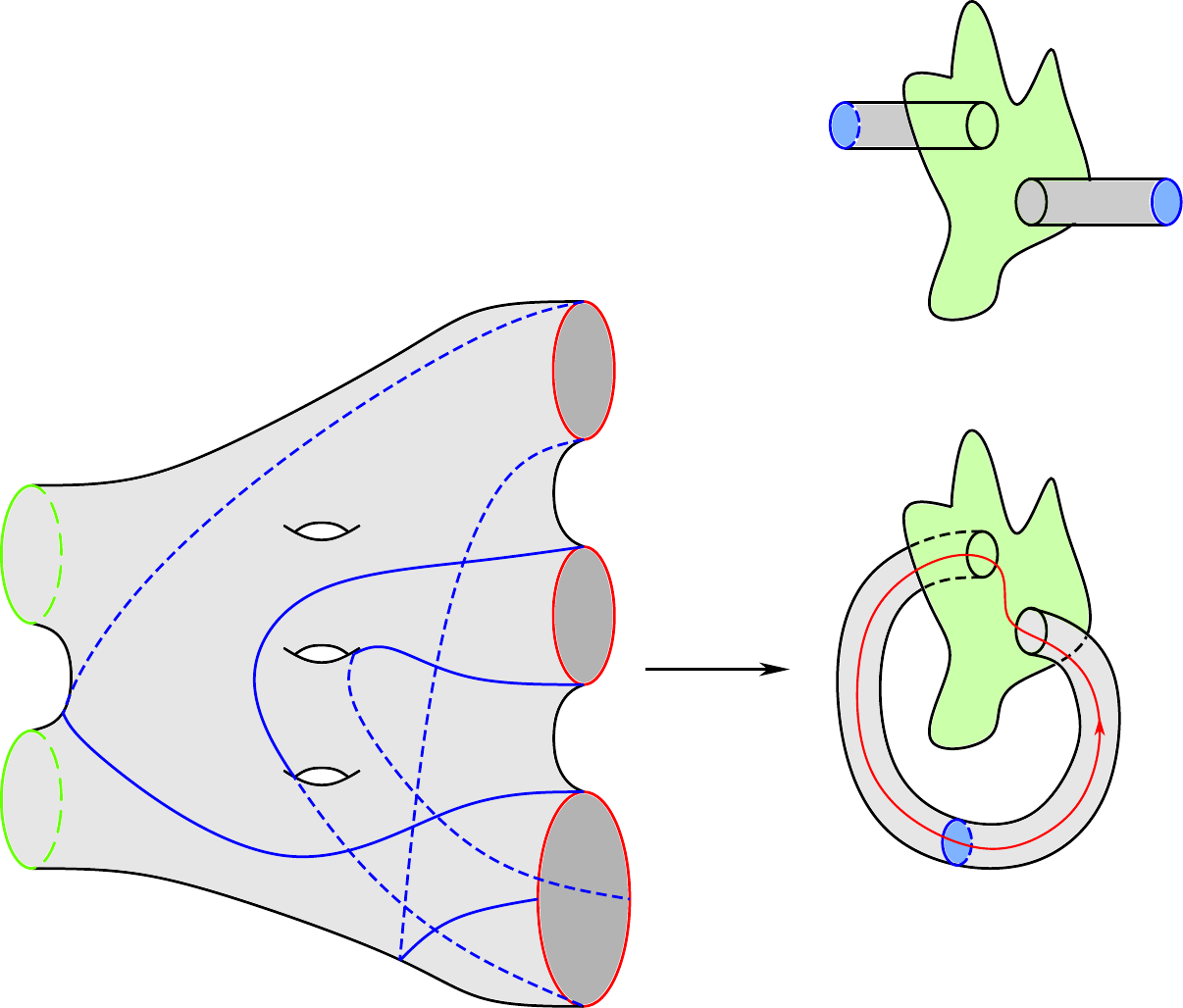}
	\caption{$F=f^{-1}(X_C)$ is a set of embedded disjoint proper arcs in the $w$-admissible surface $S$ after applying Lemma \ref{lemma: no loop}. After cutting, $S\setminus F$ maps into the thickened vertex space $V$, which deformation retracts to $X_A$.}\label{fig: graphofspace}
\end{figure}

Now $S\setminus F$ maps into the thickened vertex space $V$ and hence each polygonal boundary represents a conjugacy class in $A$, referred to as its \emph{winding class}.

\begin{definition}[(simple) normal form, disk-pieces, and annulus-pieces]\label{def: normal form}
	We refer to each component of $S\setminus F$ as a \emph{piece}. Such a decomposition of $S$ into pieces is called a \emph{normal form} of $S$. A normal form is \emph{simple} if each piece has exactly one polygonal boundary and is either a disk or an annulus, depending on whether the winding class of the unique polygonal boundary is trivial.

	We refer to the two kinds of pieces in a simple normal form as \emph{disk-pieces} and \emph{annulus-pieces} based on their topological type; see Figure \ref{fig: piecesturns} an illustration of such pieces.
\end{definition}

We can always simplify $S$ so that it admits a simple normal form.

\begin{lemma}\label{lemma: simple normla form}
	For any $w$-admissible surface $S$, there is a $w$-admissible surface $S'$ with $\deg(S')=\deg(S)$ and $-\chi(S')\le -\chi(S)$ so that $S'$ admits a simple normal form. Moreover, $S'$ can be chosen to be boundary-incompressible if $S$ is.
\end{lemma}
\begin{proof}
	By the discussion above, we may simplify $S$ so that it admits a normal form, which may not be simple in general.
	Note that each piece of $S$ has at least one polygonal boundary since each component of $S$ has nonempty $w$-boundary and $F$ contains no embedded loop.
	Suppose $S$ has a piece $P$ with at least two polygonal boundary components. Then $\chi(P)\le0$.
	Cut out a collar neighborhood of each polygonal boundary and remove the remaining part of $P$ to obtain a new surface $S'$. The part ignored has the same homotopy type as $P$ and hence has non-positive Euler characteristic. Hence $-\chi(S')\le -\chi(S)$.
	Up to homotopy we may assume the non-polygonal boundary of each collar neighborhood is mapped to the vertex space $X_A$. This makes $S'$ a $w$-admissible surface with $\deg(S')=\deg(S)$.
	The same procedure can be done if $P$ has exactly one polygonal boundary with $\chi(P)<0$.
	If there is an annulus piece $P$ where the polygonal boundary has trivial winding class, then the other boundary is a null homotopic loop in $X_A$, which we cap it off and decreases the negative Euler characteristic.
	Repeating the procedures above we arrive at a desired $w$-admissible surface $S'$ in simple normal form.
	
	Note that in the simplifying process above, including that in the proof of Lemma \ref{lemma: no loop}, $S'$ is obtained from $S$ by three kinds of modifications: homotope the map $f$ to $X$; cut along a loop representing a conjugacy class in $A$ and restrict the map to $X$ to a subsurface of the resulting surface; cut along a loop that is null homotopic in $X$ and fill in with two disks.
\end{proof}

\subsection{The structure of a polygonal boundary}\label{subsec: polygonal boundary}
Suppose as above that $w$ is written as a cyclically reduced word $w=a_1 t^{e_1}\cdots a_k t^{e_k}$ with $e_i=\pm 1$ and $a_i\in A$, represented by a tight loop $\gamma$ in $X$ corresponding to this expression. Recall from Section \ref{subsec: setup} that the edge space $X_C$ cuts $\gamma$ into $k$ segments $\gamma_1,\dots, \gamma_k$, equipped with the orientation induced from $\gamma$. The segments with the reversed orientation are denoted as $\bgamma{1},\dots,\bgamma{k}$. Also recall that segments fall into four types, $PP$, $PN$, $NP$, $NN$, depending on which side of $X_C$ the segment starts and ends at.

Fix a (disk- or annulus-)piece. Its unique polygonal boundary has an induced orientation. By definition, every other side of the polygonal boundary is a copy of some $\gamma_i$ or $\bgamma{i}$ (depending on whether the $w$-boundary it lies on is positive or negative). We refer to these sides as \emph{arcs}.
The other half of the sides are proper arcs in $F=f^{-1}(X_C)$, which we call \emph{turns}, each starting from an arc $\alpha=\gamma_i^{\pm1}$ to another arc $\alpha'=\gamma_j^{\pm1}$ for some $i,j$. By our choice of $\gamma$ and $\tau$, such a turn as a path in $X_C$ starts and ends at the base point $b_C$ and hence is a based loop representing some element $c\in C$, referred to as the \emph{winding number}. We encode the type of each turn as an ordered triple $(\alpha,c,\alpha')$.

Recall that each piece is mapped to the thickened vertex space $V$, and hence each turn is either on the positive side or the negative side of $X_C$. If a turn has type $(\alpha,c,\alpha')$ and lies on the positive side, then $\alpha$ must end on the positive side and $\alpha'$ must start from the positive side. Similarly if such a turn lies on the negative side. In particular, not every ordered pair of arcs $(\alpha,\alpha')$ can appear in the triple describing the type of a turn.


There is a pairing of turns in the normal form of a $w$-admissible surface as pieces are glued together along turns. Two paired turns are on the opposite sides of $X_C$.
The type of a turn determines the type of its paired turn. For instance, a turn of type $(\gamma_i,c,\gamma_j)$ must be paired with a turn of type $(\gamma_{j-1},c^{-1},\gamma_{i+1})$, indices taken mod $k$, where the winding number becomes its inverse due to the opposite orientation induced from the two pieces; see Figure \ref{fig: piecesturns}. 
We say two such turn types are \emph{paired}.
\begin{figure}
	\labellist
	\small \hair 2pt
	\pinlabel $\gamma_k$ at 6 60
	\pinlabel $\gamma_j$ at 60 93
	\pinlabel $\gamma_i$ at 60 27
	
	\pinlabel $c$ at 95 60
	\pinlabel $c^{-1}$ at 117 62
	
	\pinlabel $\gamma_{j-1}$ at 142 80
	\pinlabel $\gamma_{i+1}$ at 142 40
	
	\endlabellist
	\centering
	\includegraphics[scale=1]{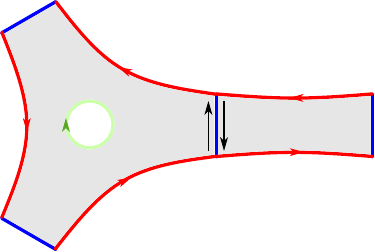}
	\caption{An annulus-piece (left) and a disk-piece (right) glued along paired turns that are of types $(\gamma_i,c,\gamma_j)$ and $(\gamma_{j-1},c^{-1},\gamma_{i+1})$ for some $c\in C$.}\label{fig: piecesturns}
\end{figure}
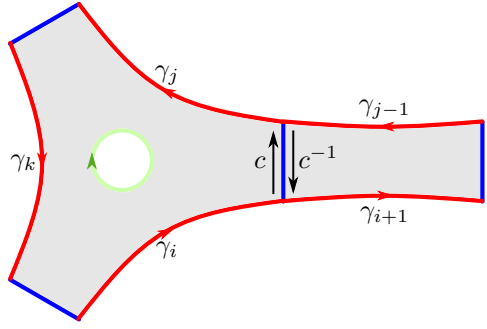

In below are some basic observations in relation to some crucial assumptions we made. The first is related to the tightness of $\gamma$.
\begin{lemma}\label{lemma: one turn}
	The polygonal boundary of any disk-piece in a simple normal form of a $w$-admissible surface has at least two turns.
\end{lemma}
\begin{proof}
	It the polygonal boundary has only one turn, then it has only one arc as well. Suppose the arc is a copy of $\gamma_i$. In the cyclically reduced expression $w=a_1 t^{e_1}\cdots a_k t^{e_k}$, the segment $\gamma_i$ corresponds to $a_i\in A$. As the two ends of $\gamma_i$ are connected by this turn, they lie on the same side of $X_C$, which means $e_i=-e_{i+1}$, and thus we must have $a_i\notin C$.
	Now the disk-piece provides a homotopy between $\gamma_i$ and the turn relative to the endpoints. Since the turn is a loop in the edge space $X_C$, we have $a_i\in C$, leading to a contradiction.
\end{proof}

The second is an interpretation of the boundary-incompressibility of admissible surfaces.
\begin{lemma}\label{lemma: no backtracking}
	In the normal form of a $w$-admissible boundary-incompressible surface $S$, there is no turn of type $(\gamma_i,id,\bgamma{i})$ or $(\bgamma{i},id,\gamma_i)$ for any $i$.
\end{lemma}
\begin{proof}
	If there were such a turn of type $(\gamma_i,id,\bgamma{i})$, then it is a proper arc going from a positive $w$-boundary of $S$ representing $w^n$ to a negative one representing $w^{-m}$ for some $m,n\in\Z_+$. The union of these two boundary components with this proper arc has a collar neighborhood $\Sigma\subset S$ that is a pair of pants. The fact that the winding number of the turn is $id\in C$ implies that the third boundary of $\Sigma$ represents $w^{m-n}$ (with the orientation induced from $\Sigma$). This contradicts the boundary-incompressibility of $S$.
\end{proof}

\subsection{Possible pieces}\label{subsec: possible pieces}
For a $w$-admissible surface in simple normal form, we know by definition it consists of disk-pieces and annulus-pieces. We define explicitly a set $\Pcal$ of disk-pieces and annulus-pieces, which include all pieces that may appear in a simple normal form of some boundary-incompressible $w$-admissible surface. The pieces in $\Pcal$ a priori may not come from a $w$-admissible surface.

To describe a piece in $\Pcal$, we start by constructing a map into $X$ from an oriented circle divided into $2n$ sides for some $n\ge1$, which will be the polygonal boundary. Label the sides in a cyclic sequence $(s_1,\dots, s_{2n})$. For any $1\le j\le n$, let $s_{2j-1}$ be a copy of $\gamma_{i_j}^{e_j}$ with $e_j=\pm 1$ so that it serves as an arc, and let $s_{2j}$ be mapped to a loop in $X_C$ based at $b_C$ representing some $c_j\in C$ to serve as a turn of type $(\gamma_{i_j}^{e_j},c_j,\gamma_{i_{j+1}}^{e_{j+1}})$.

There are two requirements. Firstly, each turn $s_{2j}$ is on one side of $X_C$: Either $\gamma_{i_j}^{e_j}$ ends on the positive side and $\gamma_{i_{j+1}}^{e_{j+1}}$ starts from the positive side so that $s_{2j}$ is on the positive side, or $s_{2j}$ is on the negative side defined in a similar manner.
Secondly, if $i_j=i_{j+1}$ and $e_j=-e_{j+1}$, then we require $c_j\neq id_C$ as they are ruled out by boundary-incompressibility as in Lemma \ref{lemma: no backtracking}.

We say a turn type is \emph{admissible} if it satisfies both requirements. Denote by $\Tcal$ the set of admissible turn types.

We say such a circle with the map described above satisfying both requirements is \emph{an abstract polygonal boundary}, which defines a loop in $X$. As a consequence of the first requirement, this loop naturally shrinks to a loop in the thickened vertex space $V$ and further to a loop in $X_A$. Hence each abstract polygonal boundary represents a conjugacy class in $A$, which we refer to as \emph{the winding class}. 

The assumption that $\gamma$ is tight implies that any abstract polygonal boundary with only one turn (and one arc) has nontrivial winding class, similar to Lemma \ref{lemma: one turn}.

Now we construct an abstract piece in $\Pcal$ for any given abstract polygonal boundary. If the winding class is trivial, think of the underlying circle as the boundary of a disk, then the map extends to the interior of the disk. This disk with the map into $X$ is an abstract disk-piece in $\Pcal$. If the winding class is nontrivial, consider an annulus where one of the boundary circle is the abstract polygonal boundary and the other is a loop in $X_A$ whose inverse represents the winding class. The map on the annulus is a homotopy, which defines an annulus-piece in $\Pcal$.

The set $\Pcal$ is the set of all abstract disk- or annulus-pieces. Clearly by Lemma \ref{lemma: no backtracking}, the polygonal boundary of a genuine disk-piece or annulus-piece has the structure of an abstract polygonal boundary, and the notion of the winding class agrees. Thus each piece that appears in a simple normal form of some boundary-incompressible $w$-admissible surface lies in $\Pcal$. It is not important to us but it seems that (in all known cases) all pieces in $\Pcal$ appear this way: The strategy is to glue finitely many abstract pieces together to close up all corners, but how exactly it works out needs a case-by-case analysis which we do not pursue here.


\subsection{The gluing graph and Euler characteristic}
For any admissible surface $S$ in normal form, there is a \emph{gluing graph} $\Gamma_S$ that encodes how the surface decomposes into pieces. Each vertex of $\Gamma$ represents a piece in the normal form and each edge represents a gluing along paired turns of two pieces. By Mayer--Vietoris, we have
$$\chi(S)=\sum_v \chi(v)- \# e,$$
where the summation is taken over all vertices $v$ of $\Gamma_S$, $\chi(v)$ is the Euler characteristic of the piece corresponding to $v$, and $\# e$ is the number of edges in $\Gamma_S$.

When $S$ is decomposed in \emph{simple} normal form, each piece is either a disk or an annulus. Hence $v_d\defeq \sum_v \chi(v)$ is the number of disk-pieces. 

Note that each edge $e$ glues two turns together, so $2\#e$ is the total number of turns. Since on each polygonal boundary, half of the sides are turns and the other half are arcs, the total number of turns is also the total number of arcs. Recall that each copy of the tight loop $\gamma$ representing $w$ is cut into $|w|$ arcs, so the total number of arcs is $\deg(S)\cdot |w|$. Hence
$$2\# e=\# \text{turns}=\deg(S)\cdot |w|.$$

The following lemma summarizes the calculations above.
\begin{lemma}\label{lemma: Euler}
    For any $w$-admissible surface $S$ in simple normal form, we have
    \begin{equation}\label{eqn: Euler}
        -\chi(S)=\frac{1}{2}\deg(S)\cdot |w| - v_d,
    \end{equation}
    where $v_d$ is the total number of disk pieces in $S$.
\end{lemma}

\section{The LP-duality method}\label{sec: duality}
In Theorems \ref{thmA: less tech main} and \ref{thmA: main}, the goal is to establish a lower bound of $-\chi(S)$ by a multiple of $\deg(S)$ for all boundary-incompressible $w$-admissible surfaces $S$, which we may assume to be in a simple normal form by Lemma \ref{lemma: simple normla form}. In view of formula (\ref{eqn: Euler}), this is equivalent to proving an upper bound of $v_d$ by a multiple of $\deg(S)$.

We prove such inequalities using a method analogous to the weak duality of linear programming. 
The method was originally developed by the author to prove uniform lower bounds (called spectral gaps) of stable commutator lengths; see \cite[Section 6.3]{Chen:sclBS} and \cite[Section 3.2]{CH:sclgap}.
We adapt the approach to our setting in this section, which comes down to the construction of a cost function meeting certain requirements. Theorems \ref{thmA: less tech main} and \ref{thmA: main} essentially follow from Theorem \ref{thm: HNN main}, which we prove by constructing a suitable cost function.

Given a boundary-incompressible $w$-admissible surfaces $S$ in simple normal form, we can count the total number $t_{T}\in\Z_{\ge0}$ of turns that have a given type $T\in\Tcal$, which is nonzero for finitely many turn types by compactness. The collection of numbers $(t_{T})_{T\in\Tcal}$ satisfies a \emph{gluing condition}, namely, $t_T=t_{T'}$ if $T$ and $T'$ are paired turn types, since each turn of type $T$ is glued to a turn of type $T'$ in $S$ when pieces are glued together.

A cost function on turns is a map $c: \Tcal\to \R$ that assigns a value to each admissible turn type in $\Tcal$. This naturally induces a cost function on the set $\Pcal$ of possible pieces. Namely, for each $P\in\Pcal$, the value $c(P)$ is the sum of $c(\alpha)$ over all turns $\alpha$ on the polygonal boundary of $P$ and $c(\alpha)$ is set to be $c(T)$ if $T\in \mathcal{T}$ is the type of the turn $\alpha$.

We are interested in cost functions meeting two requirements, one relating the cost to $v_d$, the total number of disk pieces, and the other relating the cost to the degree $\deg(S)$.

\begin{lemma}\label{lemma: cost and disk number}
    For a cost function $c: \Tcal\to \R$, if the induced cost function on possible pieces satisfies $c(P)\ge \chi(P)$ for any $P\in \Pcal$, then using the notation above we have
    $$\sum_{T\in\Tcal} c(T) t_T\ge v_d$$
    for any boundary-incompressible $w$-admissible surfaces $S$ in simple normal form.
\end{lemma}
\begin{proof}
    Note that $\chi(P)$ is either one or zero, depending on whether $P$ is a disk-piece or annulus-piece. Hence its sum over all pieces $P$ in the simple normal form is exactly $v_d$, the number of disk pieces. Hence by assumption we have
    $$\sum_P c(P) \ge v_d,$$
    where the sum is taken over all pieces in the simple normal form of the surface $S$.
    
    On the other hand, by definition $c(P)$ is itself the sum of $c(\alpha)$ over all turns $\alpha$ that appear in the piece $P$. By collecting turns of the same type, we see that 
    $$\sum_P c(P) =\sum_{T\in\Tcal} c(T) t_T.$$
    Hence the desired inequality follows.
\end{proof}

\begin{proposition}\label{prop: LP dual}
    If a cost function $c: \Tcal\to \R$ satisfies the requirement in Lemma \ref{lemma: cost and disk number} and $\sum_{T\in\Tcal} c(T) t_T=\lambda \deg(S)$ for any boundary-incompressible $w$-admissible surfaces $S$ in simple normal form, where $\lambda$ is a constant independent of $S$ (but possibly depending on $w$ or the underlying group), then 
    $$\lambda \deg(S)\ge v_d.$$
    As a consequence, we have 
    $$-\chi(S)\ge \left(\frac{|w|}{2}-\lambda\right)\deg(S)$$
    for all boundary-incompressible $w$-admissible surfaces $S$.
\end{proposition}
\begin{proof}
    The first inequality is evident by the assumption and Lemma \ref{lemma: cost and disk number}. Combining this with formula (\ref{eqn: Euler}) we obtain the second inequality for any boundary-incompressible $w$-admissible surfaces $S$ in simple normal form. For a general boundary-incompressible $w$-admissible surfaces $S$, we can put it into simple normal form by Lemma \ref{lemma: simple normla form}.
\end{proof}

\begin{remark}
	It might be unclear at first glance if there are cost functions with $\sum_{T\in\Tcal} c(T) t_T =\lambda \deg(S)$ for all $S$. Actually, there are many such functions. To see one of them, note that $\deg(S)$ is a linear function in variables $(t_T)_{T\in \Tcal}$ since the degree is a constant multiple (independent of $S$) of the total number of turns. 
	To get more such functions, note that changing $c(T)$ and $c(T')$ leaving $c(T)+c(T')$ invariant does not change $\sum_{T\in\Tcal} c(T) t_T$, for any paired turn types $T,T'\in\mathcal{T}$.	
\end{remark}

\section{A lower bound of the minimal complexity}\label{sec: lower bound}
As in Section \ref{sec: normal form}, let $H=A\star_C$ be the HNN extension associated to injections $i_P,i_N:C\inj A$. In this section, we focus on a cyclically reduced word $w$ taking the special form $$w=a_1t^{-1} b_1 t a_2 t^{-1} b_2 t \cdots a_m t^{-1} b_m t x t \in H,$$ 
where $m\ge1$, $x\in A$, $a_i\in A\setminus i_P(C)$ and $b_i\in A\setminus i_N(C)$ under the standard presentation (\ref{eqn: std presentation}). The goal is to prove Theorem \ref{thm: HNN main} below, establishing a lower bound of the minimal complexity of $w$-admissible boundary-incompressible surfaces. It is a somewhat standard trick (Lemma \ref{lemma: trick}) to reduced the case of a general word (Theorem \ref{thmA: main}) with $t$-exponent sum $\pm 1$ to this special case.

The assumptions of Theorem \ref{thm: HNN main} involve two conditions, which we now introduce. 
\begin{definition}\label{def: n-RF}
	Given a subgroup $C\le A$, for some $2\le n\le \infty$, an element $a\in A\setminus C$ is \emph{length-$n$ relatively free to $C$ ($n$-RF)} if 
	$a^{e_1}c_1\cdots a^{e_k}c_k\neq id$
	in $A$ for any $k\in\Z_+$, $e_i=\pm1$, and $c_i\in C$, provided that 
	\begin{enumerate}
		\item $c_i\neq id$ for any $i$ with $e_i=-e_{i+1}$ (indices taken mod $k$), and
		\item there are no $n$ $e_i$'s of the same sign.
	\end{enumerate}

	We say the pair $(A,C)$ is $n$-RF if $a$ is $n$-RF rel $C$ for all $a\in A\setminus C$.
\end{definition}

If $C$ is the trivial subgroup, then $a\in A\setminus C$ is $n$-RF rel $C$ if and only if $a$ has order at least $n$.

Roughly speaking, the $n$-RF condition requires that there is no short relation (measured by the quantifier $n$) among $a$ and $C$.
In particular, if the subgroup generated by $a$ and $C$ is isomorphic to $\langle a\rangle\star C$, where $\langle a\rangle$ is the cyclic subgroup of $A$ generated by $a$, then $a$ is $n$-RF, where $2\le n\le\infty$ is the order of $a$.
Taking $n=\infty$, it is easy to see that $a$ is $\infty$-RF if and only if the subgroup generated by $a$ and $C$ is (naturally) isomorphic to $\Z\star C$, and such $a$ is said to be \emph{free relative to $C$}; see \cite[Theorem 4.1]{FennRourke}.


A weaker condition only restricts relations in which all exponents of $a$ have the same sign.
\begin{definition}\label{def: n-RTF}
	Given a subgroup $C\le A$ and $2\le n\le\infty$, we say $a\in A\setminus C$ is \emph{$n$-relatively torsion-free ($n$-RTF)} in the group-subgroup pair $(A,C)$ if $ac_1\cdots a c_k\neq id$ for any $c_i\in C$ and any $1\le k<n$. Note that this is automatically true if $n=2$ as $a\notin C$.
	
	We say the pair $(A,C)$ is $n$-RTF if $a$ is $n$-RTF for all $a\in A\setminus C$.
\end{definition}

Clearly if $a\in A\setminus C$ is $n$-RF rel $C$ then it is also $n$-RTF.

If $C$ is a normal subgroup, then $a\in A\setminus C$ is $n$-RTF rel $C$ if and only if its image in $A/C$ has order at least $n$. In particular, when $C$ is trivial, being $n$-RF and $n$-RTF are equivalent.

The $n$-RTF condition holds in many examples (even for pairs $(A,C)$), for instance maximal cyclic subgroups are $\infty$-RTF in surface groups and right-angled Artin groups \cite[Example 3.14 and Lemma 3.15]{CH:sclgap}. It also has nice inheritance properties in the context of graphs of groups and graph products; see \cite[Section 3.4 and Lemma 5.4]{CH:sclgap} for more examples and details on this condition.

\begin{theorem}\label{thm: HNN main}
    With the notation above, for $w=a_1t^{-1} b_1 t a_2 t^{-1} b_2 t \cdots a_m t^{-1} b_m t x t \in H=A\star_C$, suppose for some $2\le n\le\infty$ we have:
    \begin{enumerate}
        \item $a_1$ is $n$-RF rel $i_P(C)$ and $b_m$ is $n$-RF rel $i_N(C)$, and
        \item each $a_i$ is $n$-RTF in $(A, i_P(C))$ and each $b_i$ is $n$-RTF in $(A,i_N(C))$ for all $1\le i\le m$. 
    \end{enumerate}
    Then for any $w$-admissible boundary-incompressible surface $S$, we have
    $$-\chi(S)\ge \left(1-\frac{1}{n}\right)\deg(S).$$

\end{theorem}

It is worth noting that only the $n$-RTF assumption is needed to prove the analogous estimate (\cite[Theorem 3.8]{CH:sclgap}) in the context of stable commutator length in a graph of groups. 

The following corollary explains how estimates of the complexity of $w$-admissible surfaces can be applied to obtain injectivity of subgroups under the quotient map. This corollary slightly generalizes a result of Fenn--Rourke \cite[Theorem 4.1]{FennRourke} carefully explaining and generalizing Klyachko's method \cite{Klyachko}: In their statement each $a_i$ (resp. $b_i$) is assumed to be $\infty$-RF rel $i_P(C)$ (resp. $i_N(C)$), while we only need this for $a_1, b_m$ and the weaker $\infty$-RTF condition on the other $a_i$'s and $b_i$'s; see Example \ref{example: weaker assumption} below.
Klyachko's Theorem \ref{thm: Klyachko} and other Freiheitssatz theorems quickly follow from this result after applying a standard algebraic trick (Lemma \ref{lemma: trick}), which we explain in Section \ref{sec: app}.

\begin{corollary}\label{cor: injetivity for special w}
	For the HNN extension $H=A\star_{C}$ and the word $w$ satisfying the assumptions in Theorem \ref{thm: HNN main} with $n=\infty$, the natural map $A\to H/\llangle w\rrangle$ induced by the inclusion $A\inj H$ is injective.
\end{corollary}
\begin{proof}
	Suppose the natural map is not injective, that is, there is some $a\neq id \in A$ that lies in $\llangle w \rrangle$. As in Example \ref{example: adm surf from relations} (with $H=A\star\Z$), this gives rise to an equation (\ref{eqn: relation}), which provides a $w$-admissible surface $S$ of degree $\deg(S)=\sum_{i=1}^k |n_i|\ge k$ with $-\chi(S)=k-1$ (as it is a sphere with $k+1$ disks removed) for some $k\in\Z_+$ and $n_i\in\Z\setminus\{0\}$. Moreover, as explained in Example \ref{example: adm surf from relations}, when $k$ is minimal among all equations of this form, $S$ is boundary-incompressible. Hence by Theorem \ref{thm: HNN main} (with $n=\infty$), we have $k-1=-\chi(S)\ge\deg(S)\ge k$, which leads to a contradiction. Thus the natural map must be injective.
\end{proof}

\begin{example}\label{example: weaker assumption}
	As a simple example distinguishing the $n$-RF and $n$-RTF conditions, consider $G=\Z^2$ with standard generators $x,y$ and $C=\langle y\rangle$. Then $x$ is $\infty$-RTF rel $C$ but it is not $n$-RF rel $C$ for any $n\ge2$ due to the relation $xyx^{-1}y^{-1}=id$.
	
	Now consider the HNN extension $H=A\star_{\Z}$ with $A=\Z^2\star \Z=\langle x,y,z \mid xy=yx\rangle$, where $i_P$ and $i_N$ take a chosen generator of $\Z$ to $x$ and $y$ respectively. Then the word $w=zt^{-1}x t y t^{-1} z t^2$ satisfies our assumptions in Corollary \ref{cor: injetivity for special w} since $z$ is free relative to both $\langle x\rangle$ and $\langle y\rangle$, and $x$ is $\infty$-RTF rel $\langle y\rangle$ and similarly exchanging $x$ and $y$. However, just as in the $\Z^2$ case, $x$ is not even $2$-RF rel $\langle y\rangle$, so the assumptions in \cite[Theorem 4.1]{FennRourke} do not hold in this case.
\end{example}

In the rest of this section, we prove Theorem \ref{thm: HNN main} using the LP-duality method introduced in Section \ref{sec: duality}. We first define a cost function $c:\mathcal{T}\to \R$ and then verify the desired properties.

Note that in this case the tight loop $\gamma$ corresponding to $w$ is decomposed into $|w|=2m+1$ arcs, which we denote suggestively by $\gamma_1=a_1, \gamma_2=b_1,\cdots, \gamma_{2m-1}=a_m, \gamma_{2m}=b_m, \gamma_{2m+1}=x$. Denote the arcs on $\bar{\gamma}$ by $\gamma_{2m+1}^{-1}=x^{-1}, \gamma_{2m}^{-1}=b^{-1}_m, \gamma_{2m-1}^{-1}=a^{-1}_m,\cdots, \gamma_{2}^{-1}=b^{-1}_1, \gamma_{1}^{-1}=a^{-1}_1$. The arcs $a_i^{\pm 1}$ (resp. $b_i^{\pm 1}$) are of type $PP$ (resp. $NN$), and the arc $x$ (resp. $x^{-1}$) is the only arc of type $PN$ (resp. $NP$).

A key observation here is that there is no admissible turn going from any $a_i^{\pm 1}$ to $b_j^{\pm 1}$ or vice versa. This is indicated in the directed graph in Figure \ref{fig: turngraph} for $m=2$, where any admissible turn type $(\gamma_i^{\pm 1},\kappa,\gamma_j^{\pm 1})$ for some $\kappa\in C$ has the ordered pair $(\gamma_i^{\pm 1},\gamma_j^{\pm 1})$ represented as an oriented edge.

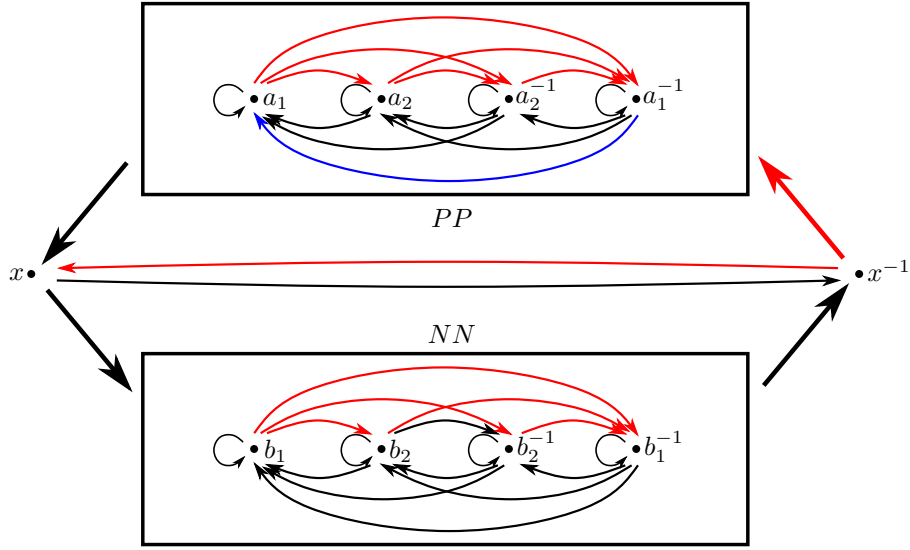
\begin{figure}
	\labellist
	\small \hair 2pt
	\pinlabel $x$ at -5 128
	\pinlabel $x^{-1}$ at 406 130
	
	\pinlabel $a_1$ at 117 210
	\pinlabel $a_2$ at 176 210
	\pinlabel $a_2^{-1}$ at 240 213
	\pinlabel $a_1^{-1}$ at 300 213
	
	\pinlabel $b_1$ at 117 45
	\pinlabel $b_2$ at 176 45
	\pinlabel $b_2^{-1}$ at 240 47
	\pinlabel $b_1^{-1}$ at 300 47
	
	\pinlabel $PP$ at 40 210
	\pinlabel $NN$ at 40 45
	\pinlabel $\text{from }a_i$ at 30 155
	\pinlabel $\text{from }a_i^{-1}$ at 100 157
	\pinlabel $\text{to }a_i$ at 285 157
	\pinlabel $\text{to }a_i^{-1}$ at 385 155
	\endlabellist
	\centering
	\includegraphics[scale=0.8]{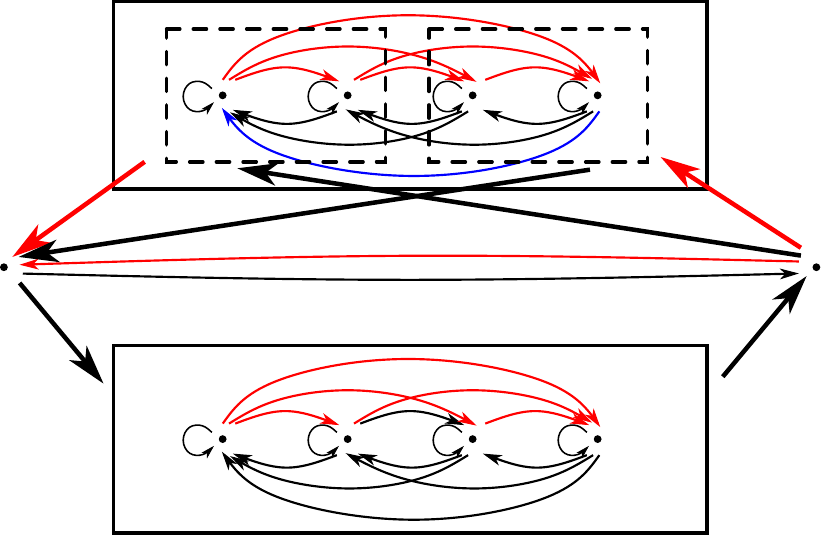}
	\caption{The directed graph encoding admissible turns between arcs, where the two solid rectangular boxes enclose all arcs of type $PP$ and $NN$ respectively. Each of the six thick big arrows represents a collection of edges connecting the vertex represented by $x$ or $x^{-1}$ with vertices in a rectangular box. Under the cost function $c$, red edges have cost $1$ when $n=\infty$, and the blue edge is the only one with negative cost. Note that $\gamma_1=a_1$, $\gamma_2=b_1$, $\gamma_3=a_2$, $\gamma_4=b_2$, and similarly for their inverses.}\label{fig: turngraph}
\end{figure}

\subsection{The cost function}
We define the cost function in a way so that the cost of an admissible turn type $(\gamma_i^{\pm 1},\kappa,\gamma_j^{\pm 1})\in\Tcal$ only depends on the ordered pair $(\gamma_i^{\pm 1},\gamma_j^{\pm 1})$. Hence we will simply define the cost $c(\gamma_i^{\pm 1},\gamma_j^{\pm 1})$ below for all $1\le i,j\le 2m+1$. To simplify the notation, we write $c_{i,j}=c(\gamma_i,\gamma_j)$, $c_{i,-j}=c(\gamma_i,\gamma_j^{-1})$, and similarly for $c_{-i,j}$ and $c_{-i,-j}$.
The cost for some pairs (say $(a_1,b_1)$) is irrelevant if the pair does not appear in any admissible turn type.

\def\arraystretch{1.5}

For $1\le i, j\le 2m+1$, we define
\begingroup
\allowdisplaybreaks
\begin{align*}
c_{i,j} &=\left\{
\begin{array}{cl}
    1-\frac{1}{n},   & i<j; \\
    \frac{1}{n},     & j\le i<2m+1;\\
    0,               & i=2m+1.
\end{array}\right.\\
c_{i,-j} &=\left\{
\begin{array}{cl}
    0,              & i=2m+1;\\
    0,              & j=2m+1;\\
    0,              & i=j=2m;\\
    1-\frac{1}{n},   & \text{otherwise}.\\
\end{array}\right.\\
c_{-i,j} &=\left\{
\begin{array}{cl}\refstepcounter{equation}\tag{\arabic{section}.\arabic{equation}}\label{eqn: cost}
    1,              & i=j=2m+1;\\
    \frac{2}{n}-1,  & i=j=1;\\
    \frac{1}{n},   & \text{otherwise}.\\
\end{array}\right.\\
c_{-i,-j} &=\left\{
\begin{array}{cl}
    1-\frac{1}{n},   & i>j; \\
    \frac{1}{n},     & i\le j<2m+1;\\
    0,               & j=2m+1.
\end{array}\right.
\end{align*}
\endgroup

In the case $n=\infty$, turns with cost $1$ are represented by red edges in Figure \ref{fig: turngraph} illustrating the case of $m=2$, where one can observe that most oriented loops contain at least one such edge.
The reason to choose this cost function becomes clearer if one restricts attention to the cost of turns among $b_i$'s (resp. $a_i$'s); see the red edges in the rectangular boxes in Figure \ref{fig: turngraph} as well as Sections \ref{subsubsec: case 1} and \ref{subsubsec: case 2} below. 

\subsection{Comparison with $\chi$}
In this section, we prove the following lemma to verify one of the conditions in Proposition \ref{prop: LP dual}. 
\begin{lemma}\label{lemma: comparison with chi}
    For any piece $P\in\Pcal$ we have $c(P)\ge \chi(P)$.
\end{lemma}

Recall that the polygonal boundary of any piece $P$ is a cyclic sequence of arcs connected by admissible turns. Thus we can view it as an oriented loop in the graph $\Gamma$, where vertices are arcs and oriented edges are admissible turns, shown in Figure \ref{fig: turngraph}. 
Each oriented edge of the loop has a cost according to the definition of $c$ above. We define the cost of any oriented path (and loop) as the sum of the cost of its edges.

The goal is to show that the total cost of the loop corresponding to the polygonal boundary of $P$ is non-negative, and moreover no less than $1$ if $P$ is a disk-piece (i.e. the polygonal boundary has trivial winding class).

We have the following basic observations.
\begin{lemma}\label{lemma: three cases}
    Any oriented loop in $\Gamma$ falls into one of the following three types:
    \begin{enumerate}
        \item It is a loop supported on $b_i^{\pm 1}$'s.
        \item It is a loop supported on $a_i^{\pm 1}$'s.
        \item It passes through both $x$ and $x^{-1}$, and it contains at least one path from $x^{-1}$ to $x$ through $a_i^{\pm1}$'s.
    \end{enumerate}
\end{lemma}
\begin{proof}
    As we observed earlier, $x^{-1}$ is the only arc of type $NP$ and $x$ is the only arc of type $PN$, and there is no oriented edge from $a_i^{\pm 1}$ to $b_j^{\pm 1}$ and vice versa; see Figure \ref{fig: turngraph}. Thus either the loop passes through both $x$ and $x^{-1}$ or it is disjoint from both. 
    
    In the latter case, the loop is supported either only on $a_i^{\pm 1}$'s or only on $b_i^{\pm 1}$'s. These are the first two cases in the lemma.
    
    In the other case, there are finitely many $x$'s and $x^{-1}$'s on the loop, and they must alternate as there is no edge from any $a_i^{\pm 1}$ to $x^{-1}$ or from any $b_i^{\pm1}$ to $x$. Hence there must be a path from $x^{-1}$ to $x$ through a bunch of $a_i^{\pm 1}$'s, which is the last case of the lemma.
\end{proof}

We prove Lemma \ref{lemma: comparison with chi} by examining these three cases respectively. In the process, we will use the following basic observation repeatedly as our (only) way of using the $n$-RF condition.

\begin{lemma}\label{lemma: use n-RF}
	Consider a graph with two vertices $u$ and $v$ and all four possible distinct oriented edges. Suppose for some $2\le n\le\infty $, 
	\begin{enumerate}
		\item both edges $(u,u)$ and $(v,v)$ have cost at least $1/n$, and 
		\item the sum of the cost of $(u,v)$ and $(v,u)$ is at least $1/n$. 
	\end{enumerate}
	Then any oriented loop visiting $u$ (resp. $v$) at least $n$ times has total cost at least $1$.
\end{lemma}
\begin{proof}
	By symmetry, it suffices to consider a loop visiting $u$ at least $n$ times. Such a loop decomposes into sub-loops each visiting $u$ exactly once. There are exactly two types of such sub-loops: 
	\begin{enumerate}
		\item Either it has exactly one edge $(u,u)$, 
		\item or it starts with $(u,v)$, ends with $(v,u)$, and has $s$ copies of $(v,v)$ in the middle for some $s\ge0$.
	\end{enumerate}
	In the first case, such a sub-loop has cost at least $1/n$, and in the second case, it has cost at least $1/n+s/n\ge 1/n$. Hence each sub-loop has cost at least $1/n$ no matter the type. The number of such sub-loops in the decomposition is the number of times that the given loop visits $u$, which is at least $n$ by assumption. So the total cost of the loop is at least $1$.
\end{proof}

The observation below shows how the $n$-RTF condition is used in our proof.
\begin{lemma}\label{lemma: use n-RTF}
	Under the assumptions of Theorem \ref{thm: HNN main}, if the polygonal boundary of a piece $P$ only contains arcs of one kind (i.e only $a_i^{\pm 1}$ or $b_i^{\pm 1}$), then $c(P)\ge \chi(P)$.
\end{lemma}
\begin{proof}
	Suppose all the arcs are $b_i$ for some fixed $1\le i\le m$. Then each edge of the loop has cost $c(b_i,b_i)=1/n$ by definition. 
	The same holds for the other cases and the proof remains the same except for possible looking at the inverse of the winding class.
	Let $s\ge1$ be the length of the loop. Then up to conjugation, the winding class is
	$$b_i \kappa_1\cdots b_i \kappa_s,$$
	where each $\kappa_i\in i_N(C)$.
	The winding class is nontrivial unless $s\ge n$ by our $n$-RTF assumption. Hence either $\chi(P)=0$ and the inequality holds trivially or $c(P)=s/n\ge 1= \chi(P)$.
	Thus we have $c(P)\ge \chi(P)$ as desired for all such $P$.
\end{proof}

\subsubsection{Case \rom{1}: Only involving $b_i^{\pm1}$'s}\label{subsubsec: case 1}
\begin{lemma}\label{lemma: case1}
    Under the assumptions of Theorem \ref{thm: HNN main}, suppose the loop in $\Gamma$ corresponding to the polygonal boundary of a piece $P\in\Pcal$ is supported on $b_i^{\pm 1}$'s. Then 
    $c(P)\ge\chi(P)$.
\end{lemma}
\begin{proof}
    Define a linear order on the $b_i^{\pm 1}$'s as follows:
    $$b_1\prec b_2\prec\cdots\prec b_m\prec b_m^{-1}\prec\cdots\prec b_2^{-1}\prec b_1^{-1}.$$
    Then the cost $c$ we defined in equation (\ref{eqn: cost}) has the property that 
    $$
    c(b_i^{e_i},b_j^{e_j})=\left\{
    \begin{array}{cl}
    1-\frac{1}{n},     & \text{if }b_i^{e_i}\prec b_j^{e_j} \text{ and }(b_i^{e_1},b_j^{e_j})\neq(b_m,b_m^{-1}) \\
    \frac{1}{n},     & \text{if }b_i^{e_i}\succeq b_j^{e_j},\\
    0,               & \text{if }(b_i^{e_i},b_j^{e_j})=(b_m,b_m^{-1}),
    \end{array}\right.$$ 
    for any $1\le i,j\le m$ and $e_i,e_j=\pm1$.
    
    If the loop passes through at least two distinct vertices, then it contains one oriented edge that is ascending in the order $\prec$ and another that is descending. 
    \begin{enumerate}
        \item If an ascending edge is not $(b_m,b_m^{-1})$, then these two edges contribute $(1-1/n)+1/n=1$ to $c(P)$ and all the other edges have non-negative cost, so $c(P)\ge1\ge\chi(P)$.
        \item If all ascending edges in the loop are $(b_m,b_m^{-1})$, then all arcs on the polygonal boundary are $b_m$ and $b_m^{-1}$. In this case, the winding class takes the form
        $$b_m^{e_1} \kappa_1 b_m^{e_2}\kappa_2\cdots b_m^{e_s} \kappa_s$$
        for some $s\ge2$, $e_i=\pm1$, and $\kappa_i\in i_N(C)$.
        As turns are admissible, we have $\kappa_i\neq id_A$ whenever $e_i=-e_{i+1}$, indices taken mod $s$. By assumption, $b_m$ is $n$-RF rel $i_N(C)$, so the above word cannot be the identity unless it contains at least $n$ copies of $b_m$ (resp. $b_m^{-1}$). If the winding class is nontrivial, we have $\chi(P)=0$, so the desired inequality clearly holds; in the exceptional case where we have at least $n$ copies of $b_m$ (or $b_m^{-1}$), the total cost is at least $1$ by Lemma \ref{lemma: use n-RF} and hence no less than $\chi(P)$.
    \end{enumerate}
    
    The remaining case is when the loop visits the same vertex, say $b_i^{e_i}$, throughout. Then the result follows from Lemma \ref{lemma: use n-RTF}.
    
    Thus we have $c(P)\ge \chi(P)$ as desired for all such $P$.
\end{proof}

\subsubsection{Case \rom{2}: Only involving $a_i^{\pm1}$'s}\label{subsubsec: case 2}
Next we show
\begin{lemma}\label{lemma: case2}
    Under the assumptions of Theorem \ref{thm: HNN main}, suppose the loop in $\Gamma$ corresponding to the polygonal boundary of a piece $P\in\Pcal$ is supported on $a_i^{\pm 1}$'s. Then 
    $c(P)\ge\chi(P)$.
\end{lemma}

Similar to the proof of the previous case, we introduce a linear order on $a_i^{\pm1}$'s as follows:
$$a_1\prec a_2\prec\cdots\prec a_m\prec a_m^{-1}\prec\cdots\prec a_2^{-1}\prec a_1^{-1}.$$

Then the cost $c$ we defined in (\ref{eqn: cost}) has the property that 
$$
c(a_i^{e_i},a_j^{e_j})=\left\{
\begin{array}{cl}
1-\frac{1}{n},     & \text{if }a_i^{e_i}\prec a_j^{e_j},\\
\frac{1}{n},     & \text{if }a_i^{e_i}\succeq a_j^{e_j}  \text{ and }(a_i^{e_1},a_j^{e_j})\neq(a_1^{-1},a_1) \\
\frac{2}{n}-1,               & \text{if }(a_i^{e_i},a_j^{e_j})=(a_1^{-1},a_1),
\end{array}\right.
$$ 
for any $1\le i,j\le m$ and $e_i,e_j=\pm1$.
    
This case is slightly more complicated than the previous one since $c(a_1^{-1},a_1)=\frac{2}{n}-1$ could be negative.

Suppose the loop corresponding to the polygonal boundary of a piece $P$ as in Lemma \ref{lemma: case2} contains $s$ copies of the edge $(a_1^{-1},a_1)$, where $s\in\Z_{\ge0}$. Then the complement of these $s$ edges in the loop consists of $s$ oriented paths from $a_1$ to $a_1^{-1}$, where each path does not contain the edge $(a_1^{-1},a_1)$. 

\begin{lemma}\label{lemma: a1 to A1}
    Let $\rho$ be an oriented path from $a_1$ to $a_1^{-1}$ supported on $a_i^{\pm1}$'s so that $\rho$ does not contain the edge $(a_1^{-1},a_1)$. Then the cost $c(\rho)\ge 1-1/n$. Moreover, we have $c(\rho)\ge 2(1-1/n)$
    if $\rho$ visits any vertex other than $a_1$ and $a_1^{-1}$.
\end{lemma}
\begin{proof}
    As $\rho$ does not contain $(a_1^{-1},a_1)$, each edge in $\rho$ has non-negative cost. Since $a_1\prec a_1^{-1}$, the path $\rho$ contains at least one ascending edge, which contributes $1-1/n$ to the cost, so $c(\rho)\ge1-1/n$.
    If $\rho$ visits any vertex $a_i^{\pm 1}$ other than $a_1$ and $a_1^{-1}$, then $\rho$ contains at least two ascending edges since $a_1\prec a_i^{\pm 1}\prec a_1^{-1}$. Thus in this case we have $c(\rho)\ge 2(1-1/n)$.
\end{proof}

Now we are ready to prove Lemma \ref{lemma: case2}.
\begin{proof}[Proof of Lemma \ref{lemma: case2}]
    By the discussion above, suppose the loop corresponding to the polygonal boundary of $P$ contains $s$ copies of the edge $(a_1^{-1},a_1)$.
    
    If $s=0$, then each edge in the loop has non-negative cost. If the loop visits at least two distinct vertices, then there is an ascending edge and a descending edge with respect to the order $\prec$, so $c(P)\ge (1-1/n)+1/n=1$. If the loop keeps visiting the same vertex, say $a_i^{e_i}$, then the result follows from Lemma \ref{lemma: use n-RTF}.
    
    If $s\ge1$, consider the $s$ paths from $a_1$ to $a_1^{-1}$ obtained by removing the $s$ copies of $(a_1^{-1},a_1)$ from the loop. By Lemma \ref{lemma: a1 to A1}, we have
    $$c(P)\ge s\left(\frac{2}{n}-1\right)+s(1-\frac{1}{n})=\frac{s}{n}\ge0.$$
    Moreover, if at least one of the paths visits some vertex other than $a_1$ or $a_1^{-1}$, then 
    $$c(P)\ge \frac{s}{n}+1-\frac{1}{n}=\frac{s-1}{n}+1\ge1\ge\chi(P).$$
    So the remaining case is where the entire loop only visits $a_1$ and $a_1^{-1}$. 
    Then we are in a situation to apply Lemma \ref{lemma: use n-RF}, noting that the cost of $(a_1,a_1^{-1})$ and $(a_1^{-1},a_1)$ sums to $1/n$.
    In this case, the winding class of the polygonal boundary is
    $$a_1^{e_1}\kappa_1a_1^{e_2}\kappa_2\cdots a_1^{e_{s'}}\kappa_{s'}$$
    where $e_i=\pm 1$, $s'\ge 2s$, and $\kappa_i\in i_P(C)$.
    Since the turns are admissible, $\kappa_i\neq id_A$ if $e_i=-e_{i+1}$. Thus it is nontrivial by the $n$-RF condition on $a_1$ unless it contains at least $n$ copies of $a_1$ (resp. $a_1^{-1}$). If it is nontrivial, then $\chi(P)=0\le c(P)$. In the exceptional case, the total cost $c(P)$ is at least $1$ by Lemma \ref{lemma: use n-RF} and hence no less than $\chi(P)$.
    Hence in any case we have $c(P)\ge\chi(P)$.
\end{proof}

\subsubsection{Case \rom{3}: Involving both $x$ and $x^{-1}$}
Now we prove
\begin{lemma}\label{lemma: case3}
    Under the assumptions of Theorem \ref{thm: HNN main}, suppose the loop in $\Gamma$ corresponding to the polygonal boundary of a piece $P\in\Pcal$ passes through both $x$ and $x^{-1}$. Then 
    $c(P)\ge1\ge \chi(P)$.
\end{lemma}

As shown in Lemma \ref{lemma: three cases}, the loop corresponding to such $P$ must contain a path from $x^{-1}$ to $x$ through $a_i^{\pm1}$'s. The key is to show
\begin{lemma}\label{lemma: X to x}
    Any path from $x^{-1}$ to $x$ through $a_i^{\pm1}$'s has cost at least $1$.
\end{lemma}

To prove this, we need the following observation:
\begin{lemma}\label{lemma: cost of subpaths}
    Any path from $x^{-1}$ to $a_1^{-1}$ through $a_i^{\pm1}$'s has cost at least $1-1/n$. The same holds for any path from $a_1$ to $x$ through $a_i^{\pm1}$'s.
\end{lemma}
\begin{proof}
    Note that by the definition of the cost $c$, we have $c(x^{-1},a_i)=1/n$ and $c(x^{-1},a_i^{-1})=1-1/n$. The cost among $a_i^{\pm 1}$'s is described in Section \ref{subsubsec: case 2} using the linear order $\prec$. 
    
    Consider a path from $x^{-1}$ to $a_1^{-1}$; see Figure \ref{fig: turngraph} for an illustration. We may assume that it only visits $a_1^{-1}$ once, since otherwise it is such a path concatenated with several loops supported on $a_i^{\pm1}$'s and each such loop has non-negative cost by Lemma \ref{lemma: case2}. Then the path does not contain the edge $(a_1^{-1},a_1)$ and thus all edges involved have non-negative cost. Now the first edge of the path is either of the form $(x^{-1},a_i^{-1})$ for some $i$, or $(x^{-1},a_i)$ for some $i$. In the former case, the first edge already has cost $1-1/n$ so the cost of the entire path is no smaller. In the latter case, there must be an ascending edge among $a_i^{\pm1}$'s with respect to the order $\prec$, which has cost $1-1/n$. Hence the cost of such a path is at least $1-1/n$ in any case.
    
    The case for a path from $a_1$ to $x$ is similar (actually symmetric by reversing the orientation of the path), noting that $c(a_i,x)=1-1/n$ and $c(a_i^{-1},x)=1/n$. Thus we omit the detailed proof.
\end{proof}
\begin{proof}[Proof of lemma \ref{lemma: X to x}]
    Consider a path from $x^{-1}$ to $x$ through $a_i^{\pm1}$'s, that is, $x$ and $x^{-1}$ only appear at the two ends and all other vertices on the path are $a_i^{\pm 1}$'s. If the path is simply the edge $(x^{-1},x)$, then its cost is $c_{-(2m+1),2m+1}=1$ by the defining equation (\ref{eqn: cost}).
    
    Now we assume the path passes through some $a_i^{\pm1}$'s. Suppose the path contains $s$ copies of the edge $(a_1^{-1},a_1)$, where $s\in\Z_{\ge0}$.
    
    If $s=0$, then all edges in the path have non-negative cost. The two edges at the two ends of the path have total cost at least $1$ unless they are of the form $(x^{-1},a_i)$ and $(a_j^{-1},x)$ respectively for some $i,j$.
    In this case, there is a subpath from $a_i$ to $a_j^{-1}$. Then there must be an ascending edge as $a_i\prec a_j^{-1}$, which has cost $1-1/n$.
    Then the total cost of the entire path is at least $2/n+(1-1/n)\ge1$.
    
    If $s\ge1$, then the complement of these $s$ edges in the path consists of a subpath from $x^{-1}$ to $a_1^{-1}$, $(s-1)$ subpaths from $a_1$ to $a_1^{-1}$, and a subpath from $a_1^{-1}$ to $x$. By Lemmas \ref{lemma: a1 to A1} and \ref{lemma: cost of subpaths}, each of these subpaths has cost at least $1-1/n$. Therefore, the cost of the entire path is at least
    $$s\left(\frac{2}{n}-1\right)+(s+1)\left(1-\frac{1}{n}\right)=\frac{s-1}{n}+1\ge 1.$$
\end{proof}

Now we are ready to prove Lemma \ref{lemma: case3}.
\begin{proof}
    By the structure revealed in Lemma \ref{lemma: three cases}, the loop corresponding to $P$ alternates between paths from $x$ to $x^{-1}$ through $b_i^{\pm1}$'s and paths from $x^{-1}$ to $x$ through $a_i^{\pm1}$'s. Note that the only edge that possibly has negative cost is $(a_1^{-1},a_1)$, so the cost of each path from $x$ to $x^{-1}$ through $b_i^{\pm1}$'s is non-negative. On the other hand, any path from $x^{-1}$ to $x$ through $a_i^{\pm 1}$'s has cost at least $1$ by Lemma \ref{lemma: X to x}. Hence the total cost is no less than $1$ as desired.
\end{proof}

\subsubsection{Proof of Lemma \ref{lemma: comparison with chi}}
Putting all three cases together, we can now prove Lemma \ref{lemma: comparison with chi}.
\begin{proof}[Proof of Lemma \ref{lemma: comparison with chi}]
    Consider any piece $P\in\Pcal$. By Lemma \ref{lemma: comparison with chi}, the loop corresponding to the polygonal boundary of $P$ falls into one of three cases. By Lemmas \ref{lemma: case1}, \ref{lemma: case2} and \ref{lemma: case3}, in each case we have $c(P)\ge\chi(P)$.
\end{proof}

\subsection{The sum $\sum_{T\in \Tcal} c(T) t_T$}

Now we turn to verifying the following computation, as the other condition that we need to apply Proposition \ref{prop: LP dual}. Recall that $t_T$ is the number of turns of type $T$ for each $T\in\Tcal$. We extend it to all turn types $T$ by setting $t_T=0$ for all $T$ not admissible.
\begin{lemma}\label{lemma: sum}
    For $w$ as in Theorem \ref{thm: HNN main} and every boundary-incompressible $w$-admissible surface $S$ in simple normal form, we have 
    $$\sum_{T\in \Tcal} c(T)t_T=\left(\frac{|w|}{2}-1+\frac{1}{n}\right)\deg(S).$$
\end{lemma}

To simplify the notation, for any $1\le i,j\le 2m+1$, let $t_{i,j}=\sum_{\kappa\in C} t_{(\gamma_i,\kappa,\gamma_j)}$, $t_{i,-j}=\sum_{\kappa\in C} t_{(\gamma_i,\kappa,\gamma_j^{-1})}$, and similarly for $t_{-i,j}$ and $t_{-i,-j}$. For convenience, we also set $t_{0,j}=t_{i,0}=0$ for any $i,j\in\Z$.

Since $c(\gamma_i,\kappa,\gamma_j)=c_{i,j}$ does not depend on $\kappa$, we have
\begin{equation}\label{eqn: four terms}
    \sum_{T\in \Tcal} c(T) t_T=\sum_{i,j=1}^{2m+1} c_{i,j} t_{i,j}+\sum_{i,j=1}^{2m+1} c_{i,-j} t_{i,-j}+\sum_{i,j=1}^{2m+1} c_{-i,j} t_{-i,j}+\sum_{i,j=1}^{2m+1} c_{-i,-j} t_{-i,-j}.
\end{equation}

For the computation below, we use the following basic facts.
\begin{lemma}\label{lemma: t equality from gluing}
    We have 
    \begin{enumerate}
        \item $t_{i,j}=t_{j-1,i+1}$,
        \item $t_{i,-j}=t_{-(j+1),i+1}$,
        \item $t_{-i,j}=t_{j-1,-(i-1)}$,
        \item $t_{-i,-j}=t_{-(j+1),-(i-1)}$,
    \end{enumerate}
    for any $1\le i,j\le 2m+1$, where each $i\pm 1$ or $j\pm 1$ is interpreted mod $2m+1$.
\end{lemma}
\begin{proof}
    Recall that, by the gluing condition, we have $t_T=t_{T'}$ for paired turn types $T,T'$. A turn of type $(\gamma_i,\kappa,\gamma_j)$ is paired with $(\gamma_{j-1},\kappa^{-1},\gamma_{i+1})$ for any $\kappa\in C$; see Figure \ref{fig: piecesturns}. Taking the sum over all $\kappa\in C$ proves the first equality.
    
    The others hold for a similar reason. For instance, a turn of type $(\gamma_i,\kappa,\gamma^{-1}_j)$ is paired with $(\gamma^{-1}_{j+1},\kappa^{-1},\gamma_{i+1})$ for any $\kappa\in C$.
\end{proof}

\begin{lemma}\label{lemma: t equality from normalizing}
    For any $1\le i\le 2m+1$, we have 
    $$\sum_{j=-(2m+1)}^{2m+1} t_{i,j}=\sum_{j=-(2m+1)}^{2m+1} t_{j,i}=\frac{1}{2}\deg(S).$$
    The same holds with $i$ replaced by $-i$.
\end{lemma}
\begin{proof}
    The first summation is the total number of turns starting from $\gamma_i$, which is exactly the total number of copies of $\gamma_i$ that appear on $\partial S$, which is $\deg_+(S)$. By Lemma \ref{lemma: half degree}, we have $\deg_+(S)=\frac{1}{2}\deg(S)$ since $p(w)=1$ for the projection $p: H\to \Z$ taking the standard generator $t$ to $1$.
    
    The second summation is the total number of turns ending at $\gamma_i$ and thus is equal to the previous one.
    
    If we replace $i$ by $-i$, the same argument above holds with $\gamma_i$  and $\deg_+(S)$ replaced by $\gamma_i^{-1}$ and $\deg_-(S)$ respectively.
\end{proof}

\begin{lemma}\label{lemma: vanishing}
    For any $1\le i\le 2m$, we have $t_{i,i+1}=0$, $t_{-(i+1),-i}=0$, and $t_{2m+1,1}=t_{-1,-(2m+1)}=0$. 
\end{lemma}
\begin{proof}
    The turn type $(\gamma_i,\kappa,\gamma_{i+1})$ is not admissible for any $\kappa\in C$, since if $\gamma_i$ ends on one side of the edge space $X_C$ then $\gamma_{i+1}$ starts on the other side. This shows $t_{i,i+1}=0$. The others hold for a similar reason.
\end{proof}

Now we compute the four summations on the right hand side of equation (\ref{eqn: four terms}), starting with the first two. By the definition of the cost function in (\ref{eqn: cost}), we have
\begin{align}\label{eqn: 1/4}
\begin{split}
    \sum_{i,j=1}^{2m+1} c_{i,j}t_{i,j}  &= \sum_{i=1}^{2m}\sum_{j>i} \left(1-\frac{1}{n}\right) t_{i,j} + \sum_{i=1}^{2m}\sum_{1\le j\le i} \frac{1}{n}t_{i,j}\\
    &= \left( 1-\frac{2}{n}\right) \sum_{i=1}^{2m}\sum_{j>i} t_{i,j}
    + \frac{1}{n} \sum_{i=1}^{2m}\sum_{j=1}^{2m+1}t_{i,j},\\
    &= \left( 1-\frac{2}{n}\right) \rom{1}_1 +\frac{1}{n}\rom{2}_1,
\end{split}
\end{align}

where $\rom{1}_1=\sum_{i=1}^{2m}\sum_{j>i} t_{i,j}$ and $\rom{2}_1=\sum_{i=1}^{2m}\sum_{j=1}^{2m+1}t_{i,j}$.

We also have
\begin{align}\label{eqn: 2/4}
\begin{split}
    \sum_{i,j=1}^{2m+1} c_{i,-j}t_{i,-j}    &=\sum_{i=1}^{2m}\sum_{j=1}^{2m}\left(1-\frac{1}{n}\right)t_{i,-j} - \left(1-\frac{1}{n}\right)t_{2m,-2m}\\
    &=\frac{1}{n}\sum_{i=1}^{2m}\sum_{j=1}^{2m+1}t_{i,-j}+\sum_{i=1}^{2m}\sum_{j=1}^{2m}\left(1-\frac{2}{n}\right)t_{i,-j} -\frac{1}{n}\sum_{i=1}^{2m}t_{i,-(2m+1)}\\
    &- \left(1-\frac{1}{n}\right)t_{2m,-2m}\\
    &=\frac{1}{n}\rom{2}_2 + \left(1-\frac{2}{n}\right)\rom{1}_2 -\frac{1}{n}\rom{3}_1
    - \left(1-\frac{1}{n}\right)t_{2m,-2m},
\end{split}
\end{align}
where $\rom{1}_2=\sum_{i=1}^{2m}\sum_{j=1}^{2m}t_{i,-j}$, $\rom{2}_2=\sum_{i=1}^{2m}\sum_{j=1}^{2m+1}t_{i,-j}$, and $\rom{3}_1=\sum_{i=1}^{2m}t_{i,-(2m+1)}$.

Putting them together, we deduce
\begin{lemma}\label{lemma: first half}
$$
    \sum_{i,j=1}^{2m+1} c_{i,j}t_{i,j}+\sum_{i,j=1}^{2m+1} c_{i,-j}t_{i,-j}   =\left(1-\frac{2}{n}\right)(\rom{1}_1+\rom{1}_2)+\frac{m}{n}\deg(S)
    -\frac{1}{n}\rom{3}_1 - \left(1-\frac{1}{n}\right)t_{2m,-2m}.
$$
\end{lemma}
\begin{proof}
    Note by Lemma \ref{lemma: t equality from normalizing} we have $\sum_{j=1}^{2m+1} t_{i,j}+\sum_{j=1}^{2m+1}t_{i,-j}=\frac{1}{2}\deg(S)$ for any $1\le i\le 2m$. Hence
    $$\rom{2}_1+\rom{2}_2=2m\cdot \frac{1}{2}\deg(S)=m\deg(S),$$
    and the result follows by combining equations (\ref{eqn: 1/4}) and (\ref{eqn: 2/4}).
\end{proof}

Similarly, we compute the third and fourth summation in (\ref{eqn: four terms}).

\begin{align}\label{eqn: 3/4}
    \sum_{i,j=1}^{2m+1}c_{-i,j} t_{-i,j}=\frac{1}{n}\sum_{i,j=1}^{2m+1} t_{-i,j}+ \left(1-\frac{1}{n}\right)t_{-(2m+1),2m+1} -\left(1-\frac{1}{n}\right) t_{-1,1}.
\end{align}

\begin{align}\label{eqn: 4/4}
\begin{split}
    \sum_{i,j=1}^{2m+1}c_{-i,-j}t_{-i,-j} &=\sum_{i=1}^{2m+1}\sum_{1 \le j<i} \left(1-\frac{1}{n}\right)t_{-i,-j} + \frac{1}{n} \sum_{i=1}^{2m+1}\sum_{j=i}^{2m}t_{-i,-j}\\
    &=\sum_{i=1}^{2m+1}\sum_{1 \le j<i} \left(1-\frac{2}{n}\right)t_{-i,-j} + \frac{1}{n} \sum_{i=1}^{2m+1}\sum_{j=1}^{2m+1} t_{-i,-j} -\frac{1}{n}\sum_{i=1}^{2m+1}t_{-i,-(2m+1)}\\
    &=\left(1-\frac{2}{n}\right) \rom{1}_3 + \frac{1}{n}\sum_{i,j=1}^{2m+1} t_{-i,-j} -\frac{1}{n}\rom{3}_2,
\end{split}
\end{align}
where $\rom{1}_3=\sum_{i=1}^{2m+1}\sum_{1 \le j<i} t_{-i,-j}$ and $\rom{3}_2=\sum_{i=1}^{2m+1}t_{-i,-(2m+1)}$.

Putting these two together, we get
\begin{lemma}\label{lemma: second half}
\begin{align*}
    \sum_{i,j=1}^{2m+1}c_{-i,j} t_{-i,j}+\sum_{i,j=1}^{2m+1}c_{-i,-j}t_{-i,-j}
    &=\frac{2m+1}{2n}\deg(S) 
    +\left(1-\frac{2}{n}\right) \rom{1}_3
    -\frac{1}{n}\rom{3}_2\\
    &+\left(1-\frac{1}{n}\right)t_{-(2m+1),2m+1} -\left(1-\frac{1}{n}\right) t_{-1,1}
\end{align*}
\end{lemma}
\begin{proof}
    Note that by Lemma \ref{lemma: t equality from normalizing} we have
    $$\sum_{j=1}^{2m+1}t_{-i,j}+\sum_{j=1}^{2m+1}t_{-i,-j}=\frac{1}{2}\deg(S)$$
    for any $1\le i\le 2m+1$.
    Hence 
    $$\sum_{i,j=1}^{2m+1} t_{-i,j}+\sum_{i,j=1}^{2m+1}t_{-i,-j}=\frac{2m+1}{2}\deg(S),$$
    and the result follows by combining equations (\ref{eqn: 3/4}) and (\ref{eqn: 4/4}).
\end{proof}

We now combine Lemmas \ref{lemma: first half} and \ref{lemma: second half} to complete the computation by equation (\ref{eqn: four terms}). To simplify the results, we make one further observation.

\begin{lemma}\label{lemma: staircase}
    $$\rom{1}_1+\rom{1}_2+\rom{1}_3=\frac{1}{2}\left[(2m-1)\deg(S)+t_{2m+1,-(2m+1)}+t_{-1,1}\right].$$
\end{lemma}
\begin{proof}
    Note by Lemma \ref{lemma: t equality from gluing} we have
    $$\rom{1}_1=\sum_{1\le i<j\le 2m+1}t_{i,j} = \sum_{1\le i<j\le 2m+1} t_{j-1,i+1}$$
    and 
    $$2\rom{1}_1=\sum_{1\le i<j\le 2m+1}t_{i,j} + \sum_{1\le i<j\le 2m+1} t_{j-1,i+1} =\sum_{i=1}^{2m}\sum_{j=2}^{2m+1}t_{i,j}+\sum_{i=1}^{2m} t_{i,i+1}.$$
    Since $t_{i,i+1}=0$ for all $1\le i\le 2m$ by Lemma \ref{lemma: vanishing}, we have
    $$2\rom{1}_1=\sum_{i=1}^{2m}\sum_{j=2}^{2m+1}t_{i,j}.$$
    
    A similar computation shows
    $$2\rom{1}_3=\sum_{i=2}^{2m+1}\sum_{j=1}^{2m}t_{-i,-j}.$$
    
    Lemma \ref{lemma: t equality from gluing} also implies
    $$\rom{1}_2=\sum_{i=1}^{2m}\sum_{j=1}^{2m}t_{i,-j}=\sum_{i=2}^{2m+1}\sum_{j=2}^{2m+1}t_{-i,j}.$$
    
    Combining these, we see that
    \begin{align*}
        2(\rom{1}_1+\rom{1}_2+\rom{1}_3)&=\sum_{i=1}^{2m}\sum_{j=2}^{2m+1}t_{i,j}+
    \sum_{i=1}^{2m}\sum_{j=1}^{2m}t_{i,-j}+\sum_{i=2}^{2m+1}\sum_{j=2}^{2m+1}t_{-i,j}+
    \sum_{i=2}^{2m+1}\sum_{j=1}^{2m}t_{-i,-j}\\
    &=\sum_{i=-(2m+1)}^{2m+1}\sum_{j=-(2m+1)}^{2m+1} t_{i,j}\\
    &-\sum_{i=-(2m+1)}^{2m+1} (t_{i,1} + t_{i, -(2m+1)}) - \sum_{j=-(2m+1)}^{2m+1} (t_{2m+1,j}+t_{-1,j})\\
    &+t_{2m+1,1}+t_{-1,1}+t_{2m+1,-(2m+1)}+t_{-1,-(2m+1)}.
    \end{align*}
    
    By Lemma \ref{lemma: vanishing} we have $t_{2m+1,1}=t_{-1,-(2m+1)}=0$. Combining this and Lemma \ref{lemma: t equality from normalizing}, the equation above yields
    $$ 2(\rom{1}_1+\rom{1}_2+\rom{1}_3)=(4m+2-4)\cdot\frac{1}{2}\deg(S)+t_{-1,1}+t_{2m+1,-(2m+1)},$$
    which is clearly equivalent to the desired formula.
\end{proof}

Now we are ready to prove Lemma \ref{lemma: sum}.
\begin{proof}[Proof of Lemma \ref{lemma: sum}]
    By equation (\ref{eqn: four terms}), using Lemmas \ref{lemma: first half}, \ref{lemma: second half}, and \ref{lemma: staircase}, we have
    \begin{align*}
        \sum_{T\in\Tcal} c(T) t_T
        &=\left(1-\frac{2}{n}\right)(\rom{1}_1+\rom{1}_2+\rom{1}_3)+\frac{m}{n}\deg(S)-\frac{1}{n}(\rom{3}_1+\rom{3}_2)-\left(1-\frac{1}{n}\right)t_{2m,-2m}\\
        &+\frac{2m+1}{2n}\deg(S)+\left(1-\frac{1}{n}\right)t_{-(2m+1),2m+1}-\left(1-\frac{1}{n}\right)t_{-1,1}.\\
        &=\frac{1}{2}\left(1-\frac{2}{n}\right)\left[(2m-1)\deg(S)+t_{2m+1,-(2m+1)}+t_{-1,1}\right]\\
        &+\frac{4m+1}{2n}\deg(S)-\frac{1}{n}(\rom{3}_1+\rom{3}_2)-\left(1-\frac{1}{n}\right)t_{2m,-2m}\\
        &+\left(1-\frac{1}{n}\right)t_{-(2m+1),2m+1}-\left(1-\frac{1}{n}\right)t_{-1,1}.
    \end{align*}
    
    Note that by Lemma \ref{lemma: t equality from normalizing} we have
    $$\rom{3}_1+\rom{3}_2=\sum_{i=1}^{2m}t_{i,-(2m+1)}+\sum_{i=1}^{2m+1}t_{-i,-(2m+1)}=\frac{1}{2}\deg(S)-t_{2m+1,-(2m+1)}.$$
    Substituting it in the equation above and simplifying, we have
    \begin{align*}
        \sum_{T\in\Tcal} c(T) t_T
        &=\left(\frac{2m-1}{2}+\frac{1}{n}\right)\deg(S)\\
        &+\frac{1}{2}t_{2m+1,-(2m+1)}-\frac{1}{2}t_{-1,1}+\left(1-\frac{1}{n}\right)(t_{-(2m+1),2m+1}-t_{2m,-2m}).
    \end{align*}
    Since $t_{2m+1,-(2m+1)}=t_{-1,1}$ and $t_{-(2m+1),2m+1}=t_{2m,-2m}$ by Lemma \ref{lemma: t equality from gluing}, using $|w|=2m+1$ we obtain
    $$\sum_{T\in\Tcal} c(T) t_T
        =\left(\frac{|w|}{2}-1+\frac{1}{n}\right)\deg(S).$$
\end{proof}

\subsection{Proof of Theorem \ref{thm: HNN main}}
Now we are in a place to prove Theorem \ref{thm: HNN main}.
\begin{proof}[Proof of Theorem \ref{thm: HNN main}]
    By Lemmas \ref{lemma: comparison with chi} and \ref{lemma: sum}, we may apply the LP duality method (Proposition \ref{prop: LP dual}) with $\lambda=|w|/2-1+1/n$ using the cost function $c$ we defined above. Then Proposition \ref{prop: LP dual} shows that
    $$-\chi(S)\ge \left(1-\frac{1}{n}\right)\deg(S)$$
    for any boundary-incompressible $w$-admissible surface $S$ as desired.
\end{proof}

\subsection{The $n$-RF condition}\label{subsec: n-RF}
We give a brief discussion on the $n$-RF condition (Definition \ref{def: n-RF}) in this subsection. 

The following observations easily follow from the definition:
\begin{lemma}\label{lemma: basic n-RF}
	\leavevmode
	\begin{enumerate}
		\item If $a\in A\setminus C$ is $n$-RF rel $C$, then it is $m$-$RF$ rel $C$ for any $m\le n$.\label{item: n-RF to m-RF}
		\item For a chain of subgroup $C\le B\le A$, if $a\in A\setminus B$ is $n$-RF rel $B$, then it is also $n$-RF rel $C$.\label{item: subgroup inheritance}
		\item For a chain of subgroup $C\le B\le A$, if $(A,B)$ and $(B,C)$ are both $n$-RF for some $2\le n\le\infty$, then $(A,C)$ is also $n$-RF.\label{item: subgroup chain}
	\end{enumerate}
\end{lemma}
\begin{proof}
	Items (\ref{item: n-RF to m-RF}) and (\ref{item: subgroup inheritance}) follow directly from the definition. 
	For (\ref{item: subgroup chain}), for any $a\in A\setminus C$, if $a\in B$, then it is $n$-RF rel $C$ since $(B,C)$ is $n$-RF.
	If $a\notin B$, then it is $n$-RF rel $B$ since $(A,B)$ is $n$-RF, and hence it is $n$-RF rel $C$ by item (\ref{item: subgroup inheritance}).
\end{proof}

The next lemma reveals that the $2$-RF condition is related to malnormality of $C$. Recall that a subgroup $C\le A$ is called malnormal if $aCa^{-1}\cap C=id$ for all $a\notin C$.
\begin{lemma}\label{lemma: 2-RF}
	An element $a\in A\setminus C$ is $2$-RF rel $C$ if and only if $aCa^{-1}\cap C=id$. So $(A,C)$ is $2$-RF if and only if $C$ is malnormal in $A$.
	In particular, if $(A,C)$ is $n$-RF for some $n\ge2$, then $C$ is malnormal.
\end{lemma}
\begin{proof}
	The definition of $2$-RF only requires that $ac_1a^{-1}c_2\neq id$ when $c_1,c_2\neq id$ (and an equivalent equation with $a$ and $a^{-1}$ swapped). That is, if $ac_1 a^{-1}=c_2^{-1}$, then either $c_1$ or $c_2$ is the identity, in which case we have $c_1=c_2=id$ as they are conjugate. Hence this is equivalent to that $aCa^{-1}\cap C=id$. The other assertions easily follow from Lemma \ref{lemma: basic n-RF}.
\end{proof}

It is also easy to observe that being $n$-RF rel $C$ is really a condition on the double coset $CaC$.
\begin{lemma}\label{lemma: double coset}
	If $a\in A\setminus C$ is $n$-RF (resp. $n$-RTF) rel $C$, then so is any $\tilde{a}\in CaC$.
\end{lemma}
\begin{proof}
	Let $\tilde{a}=cac'$ and suppose $\tilde{a}^{e_1}\tilde{c}_1\cdots \tilde{a}^{e_k}\tilde{c}_k=id$ for some $k\ge1$, $e_i=\pm1$ and $\tilde{c}_i\in C$, where $\tilde{c}_i\neq id$ if $e_i=-e_{i+1}$. The goal is to show that there are at least $n$ $e_i$'s of the same sign. Replacing each $\tilde{a}$ (resp. $\tilde{a}^{-1}$) by $cac'$ (resp. $c'^{-1}a^{-1}c^{-1}$), the equation can be rewritten as $a^{e_1}c_1\cdots a^{e_k}c_k=id$, where
	\[
		c_i=\left\{\begin{array}{ll}
			c'\tilde{c}_ic			&\text{if }e_i=e_{i+1}=1;\\
			c'\tilde{c}_ic'^{-1}	&\text{if }e_i=1, e_{i+1}=-1;\\
			c^{-1} \tilde{c}_i c	&\text{if }e_i=-1, e_{i+1}=1;\\
			c^{-1} \tilde{c}_i c'^{-1}	&\text{if }e_i=e_{i+1}=-1.
		\end{array}
		\right.
	\]
	Hence when $e_i=-e_{i+1}$, we have $c_i\neq id$ as it is conjugate to $\tilde{c}_i$. Since $a$ is $n$-RF, there must be at least $n$ $e_i$'s of the same sign as desired. The same proof works for the $n$-RTF condition.
\end{proof}

Here is the main proposition of this section, which we need in Section \ref{sec: app}. Such inheritance should hold more generally for graphs of groups, but we just focus on the case of an amalgam. Similar inheritance for graphs of groups holds for the $n$-RTF condition; see \cite[Corollary 3.17 and Lemma 3.18]{CH:sclgap}.
\begin{proposition}\label{prop: inheritance}
	Consider an amalgam $G=A\star_C B$. If for some $2\le n\le \infty$ both $(A,C)$ and $(B,C)$ are $n$-RF, then $(G,A)$ and $(G,B)$ are $n$-RF.
\end{proposition}

We will prove this proposition using the following more specific statement. Recall that each element $g\in A\star_C B\setminus(A\cup B)$ can be written as a reduced word, which is an expression $g=\xtt_1\cdots \xtt_k$ for some $k\ge2$ with $\xtt_i$'s alternating between elements in $A\setminus C$ and $B\setminus C$.
\begin{lemma}\label{lemma: inheritance, tech version}
	Let $g=\xtt_{-r}\xtt_{-r+1}\cdots \xtt_{0}\cdots \xtt_{r-1}\xtt_r$ be a reduced word in $G=A\star_C B$ with $r\ge1$, where $\xtt_i\in A\setminus C$ for all $i\equiv r\mod 2$ and $\xtt_i\in B\setminus C$ otherwise.
	Suppose $\xtt_{-r}$ and $\xtt_{r}$ are both $2$-RF rel $C$, and $\xtt_0$ is $n$-RTF rel $C$ for some $2\le n\le\infty$. Then $g\in G\setminus B$ is $n$-RF rel $B$.
\end{lemma}

We need one simple observation in the proof.
\begin{lemma}\label{lemma: gcg}
	Suppose $n\ge3$ in the setting of Lemma \ref{lemma: inheritance, tech version}. Then for any $c\in C$,
	\begin{enumerate}
		\item either there is some $1\le k\le r$ such that the element $h\defeq (\xtt_0\cdots \xtt_r) c (\xtt_{-r}\cdots \xtt_0)\in G$ is represented by a reduced word
		$\xtt_0\cdots (\xtt_k c_k \xtt_{-k}) \cdots \xtt_0$ for some $c_k\in C$,\label{item: gcg case 1}
		\item or the element $h$ defined above reduces to $h=\xtt_0 c_0 \xtt_0\notin C$ for some $c_0\in C$.\label{item: gcg case 2}
	\end{enumerate}
\end{lemma}
\begin{proof}
	Note that $\xtt_i$ and $\xtt_{-i}$ lie in the same free factor for all $i$ in our setup, and $\xtt_r,\xtt_{-r}\in A\setminus C$.
	Let $c_r=c$. If $\xtt_r c_r \xtt_{-r}\in A\setminus C$, then the result holds with $k=r$ as in Case (\ref{item: gcg case 1}).
	If not, then $c_{r-1}\defeq \xtt_r c_r \xtt_{-r}\in C$. In this case, we have $h=\xtt_0\cdots\xtt_{r-1}c_{r-1} \xtt_{-r+1}\cdots \xtt_0$ and we examine whether the element $\xtt_{r-1}c_{r-1} \xtt_{-r+1}\in B$ lies in $C$.
	Continuing this process inductively, 
	\begin{itemize}
		\item either we stop at some $1\le k\le r$ by having $\xtt_k c_k \xtt_{-k}\notin C$, which gives the desired result as in Case (\ref{item: gcg case 1});
		\item or we can reduce $h$ all the way to $h=\xtt_0 c_0 \xtt_0$ for some $c_0\in C$.
	\end{itemize}
	In the second situation, we must have $\xtt_0 c_0 \xtt_0\notin C$ as desired since otherwise $\xtt_0 c_0 \xtt_0 c'=id$ for some $c'\in C$, contradicting the assumption that $\xtt_0$ is $n$-RTF for some $n\ge3$.
\end{proof}

Now we prove Lemma \ref{lemma: inheritance, tech version}.
\begin{proof}[Proof of Lemma \ref{lemma: inheritance, tech version}]
	Note that if $b\in B\setminus C$, then $\xtt_r b\xtt_{r}^{-1}$ is a reduced word in $G$ by definition. If $b\in C\setminus\{id\}$ instead, then $\xtt_r b\xtt_{r}^{-1}$ lies in $A\setminus C$ since $\xtt_r$ is $2$-RF by assumption. Similarly, for any $b\in B\setminus \{id\}$, the expression $\xtt_{-r}^{-1} b\xtt_{-r}$ is either a reduced word already or is an element in $A\setminus C$.
	
	If $n=2$, it suffices to show that $gbg^{-1}b'\neq id$ for all $b,b'\in B\setminus\{id\}$.
	By moving $\xtt_{-r}$ to the end, $gbg^{-1}b'$ is conjugate to
	$$\xtt_{-r+1}\cdots \xtt_{r-1}(\xtt_r b \xtt_{r}^{-1}) \xtt_{r-1}^{-1}\cdots\xtt_{-r+1}^{-1}(\xtt_{-r}^{-1}b'\xtt_{-r}),$$ 
	which is cyclically reduced by the observation above and thus nontrivial.
	
	If $n\ge3$, consider $w=g^{e_1}b_1\cdots g^{e_k} b_k\in G$ for some $k\ge1$, $e_i=\pm1$, and $b_i\in B$ with the property that $b_i\neq id$ when $e_i=-e_{i+1}$ (indices taken mod $k$).
	The goal is to show that $w\neq id$ assuming there are no $n$ $e_i$'s of the same sign. The cyclic sequence of $e_i$'s can be cut into (cyclic) subsequences so that all $e_i$'s in each subsequence are equal and each subsequence has maximal length with this property. Corresponding to a subsequence where all $e_i=1$, we have a subword of the form $u=g b_{i+1} \cdots g b_{i+\ell}$ for some $1\le \ell<n$.
	The word is already reduced near each $b_{i+j}$ if $b_{i+j}\notin C$. When $b_{i+j}\in C$, the word reduces as in Lemma \ref{lemma: gcg}. 
	Summarizing all cases, the subword $u$ reduces to 
	$$u=\xtt_{-r} \cdots \xtt_{-2} \xtt_{-1} \left( \prod_{j=1}^{\ell-1} w_j \right)  \xtt_0 \cdots \xtt_r b_{i+\ell}$$
	for some $0\le k_j\le r$, where
	$$w_j=\xtt_0\cdots \xtt_{k_j-1} (\xtt_{k_j} d_{j} \xtt_{-k_j}) \xtt_{-k_j+1}\cdots \xtt_{-1}, \text{ if }k_j\ge1, \quad \text{and}\quad w_j=\xtt_0 d_j, \text{ if }k_j=0,$$
	for all $1\le j<\ell$, and
	\begin{itemize}
		\item either $d_j\in B\setminus C$ and $k_j=r$, 
		\item or $d_j\in C$ and $\xtt_{k_j} d_j \xtt_{-k_j}\notin C$.
	\end{itemize}
	Note that when $k_j\ge1$, the word $w_j$ is reduced, starting in $A\setminus C$ and ending in $B\setminus C$ (when $r$ is even) or vice versa (when $r$ is odd), where we read $\xtt_{k_j} d_j \xtt_{-k_j}$ as a reduced word of length $3$ when $d_j\in B\setminus C$ and as a single letter when $d_j\in C$.
	In particular, $u$ is clearly a reduced word in $G$ if $k_j\ge1$ for all $j$, as $r$ does not depend on $j$.
	In general, in the expression $\prod_{j=1}^{\ell-1} w_j$, consider a maximal subsequence of consecutive $j$'s with $k_j=0$. That is, suppose for some $0\le j_1<j_2\le \ell$ we have $k_j=0$ for all $j_1<j<j_2$, with $j_1=0$ or $k_{j_1}\ge1$, and $j_2=\ell$ or $k_{j_2}\ge1$. Then the product $\prod_{j_1<j<j_2} w_j$ together with the subsequent letter $\xtt_0$ (from $w_{j_2}$ if $j_2<\ell$ and from the last $\xtt_0$ in the expression of $u$ if $j_2=\ell$) 
	gives an expression $\xtt_0 d_{j_1+1}\cdots  \xtt_0 d_{j_2-1} \xtt_0$ with each such $d_j\in C$. Note that this involves at most $j_2-j_1\le \ell$ copies of $\xtt_0$, so
	it yields an element in $A\setminus C$ if $r$ is even and in $B\setminus C$ if $r$ is odd, since $\xtt_0$ satisfies $n$-RTF rel $C$. It follows that $u$ is always reduced.
	Moreover, the starting $\xtt_{-r}$ (resp. the ending $\xtt_r b_{i+\ell}$) ensures that the word representing $u$ above has no cancellation with the tail of the preceding $g^{-1} b_{i}$ (resp. the head of the succeeding $g^{-1}$) by the observation above. By taking inverses, a similar reduction holds for a subsequence where all $e_i=-1$.
	It follows that $w$ is a nontrivial reduced word, so $w\neq id$.	
\end{proof}

Then we deduce Proposition \ref{prop: inheritance} from Lemma \ref{lemma: inheritance, tech version}.
\begin{proof}[Proof of Proposition \ref{prop: inheritance}]
	By symmetry, it suffices to show that $(G,B)$ is $n$-RF. For any $g\in G\setminus B$, it can be written as a word 
	$g=\xtt_{-r-1}\xtt_{-r}\cdots \xtt_{0}\cdots \xtt_{r}\xtt_{r+1}\in G$ for some $r\ge0$, 
	where $\xtt_i\in A$ if $i\equiv r\mod 2$ and $\xtt_i\in B$ otherwise, and $\xtt_i\notin C$ for all $i$ except that possibly $\xtt_{-r-1}=id$ or $\xtt_{r+1}=id$.
	We may assume $\xtt_{-r-1}=\xtt_{r+1}=id$ since $g$ is $n$-RF rel $B$ if and only if the same holds for any $g'\in BgB$ by Lemma \ref{lemma: double coset}.
	Now if $r\ge1$, then $g$ is in the form of Lemma \ref{lemma: inheritance, tech version} and our assumption implies that $\xtt_{-r},\xtt_r,\xtt_0$ satisfy the requirements. Hence $g$ is $n$-RF rel $B$. 
	
	If $r=0$, then $g=\xtt_0\in A\setminus C$, consider a word $w=g^{e_1}b_1\cdots g^{e_k}b_k$ with $k\ge1$, $b_i\in B$, and $b_i\neq id$ if $e_i=-e_{i+1}$. 
	Those $b_i$'s with the property $b_i\in B\setminus C$ (if exist) cut $w$ into subwords, each of the form $u=\xtt_0^{f_1} c_1 \xtt_0^{f_2}\cdots c_r \xtt_0^{f_{r+1}}$ for some $r\ge0$ with $c_i\in C$, $f_i=\pm1$ (which is equal to some $e_{i'}$) and $c_i\neq id$ if $f_i=-f_{i+1}$.
	Since $\xtt_0\in A\setminus C$ is $n$-RF rel $C$, we see that $u\in A\setminus C$ unless we have $n$ equal $f_i$'s, which means we have $n$ equal $e_i$'s in $w$.
	Hence if there are no $n$ equal $e_i$'s, the subwords in between those $b_i$'s with $b_i\in B\setminus C$ each lies in $A\setminus C$, so $w$ must be a nontrivial element of $G$ as desired. This completes the proof.
\end{proof}

\section{Applications}\label{sec: app}
\subsection{The Kervaire conjecture and related problems}

Now we deduce from Theorem \ref{thm: HNN main} results about a general word in an HNN extension $H=A\star_C$ with $t$-exponent sum $\pm1$. Throughout this section, let $p:H\to\Z$ be the epimorphism sending $t$ to the generator $1\in\Z$ and vanishing on $A$.

\begin{theorem}\label{thm: HNN general word}
	Let $H=A\star_C$ be an HNN extension associated to injections $i_P,i_N:C\to A$ with standard presentation (\ref{eqn: std presentation}). 
	Suppose for some $2\le n\le \infty$, the group-subgroup pairs $(A,i_P(C))$ and $(A,i_N(C))$ are $n$-RF (Definition \ref{def: n-RF}). 
	Then for any $w\in H$ with $p(w)=\pm 1$ and not conjugate to $at^{\pm1}$ for some $a\in A$, any boundary-incompressible $w$-admissible surface $S$ has
	$$-\chi(S)\ge \left(1-\frac{1}{n}\right)\deg(S).$$
\end{theorem}

%

The proof reduces the problem to the case of Theorem \ref{thm: HNN main} using a somewhat standard algebraic trick, which at least goes back to Klyachko's original proof of the Kervaire--Laudenbach conjecture for torsion-free groups \cite[Lemma 3]{Klyachko}. Such a statement for free HNN extensions and its proof can be found in \cite[Lemma 4.2]{FennRourke}. We use a similar argument to produce the desired Lemma \ref{lemma: trick} for a general HNN extension.

The trick is to express a conjugate of the given word $w$ into the special form in Theorem \ref{thm: HNN main} at the cost of
passing to a different HNN extension structure of $H=A\star_C$. To see the different HNN extension structures, for each $k\in\Z_{\ge0}$, let $A_k$ be the subgroup generated by all words $t^{-i} a t^i$ with $a\in A$ and $0\le i\le k$. Note that $A_0=A$ and $A_k$ is the free product of $k+1$ copies of $A$ amalgamated over $k$ copies of $C$ when $k\ge1$. For convenience, let $A_{-1}\defeq i_N(C)=A_0\cap tA_0 t^{-1}$. Then $H=A\star_C$ is also the HNN extension of $A_{k}$ over the subgroups $A_{k-1}$ and $t^{-1}A_{k-1}t$ for any $k\ge0$.

\begin{lemma}\label{lemma: trick}
	Let $H=A\star_C$ be an HNN extension. Let $w\in H$ be an element with $p(w)=1$, where $p: H\to\Z$ is the epimorphism mentioned above. Then either $w$ is conjugate to $at$ for some $a\in A$, or there is some $k\in \Z_+$ such that a conjugate of $w$ can be written as
	$a_1 t^{-1} b_1 t \cdots a_m t^{-1} b_m t x t$
	for some $a_i\in A_{k-1}\setminus t^{-1}A_{k-2}t$, $b_i\in A_{k-1}\setminus A_{k-2}$, and $x\in A_{k-1}$ for some $m\ge1$.
\end{lemma}
\begin{proof}
	Let $K=\ker p$, which is an amalgamated free product of infinitely many $A$'s over $C$'s, generated by elements of the form $t^{-i} a t^i$ with $a\in A$ and $i\in\Z$. 
	This can be seen by considering the infinite cyclic cover (corresponding to $K$) of the space $X$ with $\pi_1 X \cong H$ constructed in Section \ref{subsec: setup}.
	For any integers $k\le \ell$, let $A_{k,\ell}\le K$ be the subgroup generated by elements of the form $t^{-i} a t^i$ with $a\in A$ and $k\le i\le \ell$. Comparing to the notation introduced above, we have $A_k=A_{0,k}$ for all $k\in\Z_{\ge0}$. Note that $A_{k,\ell}\le A_{k',\ell'}$ if $k'\le k$ and $\ell\le \ell'$, and $t^{-i} A_{k,\ell}t^i=A_{k+i,\ell+i}$.
	
	Each $g\neq id\in K$ is contained in some $A_{k,\ell}$ where we take $k$ to be maximal and $\ell$ to be minimal with this property. We refer to $A_{k,\ell}$ as the support of $g$.
	
	For any $h\in H$, the element $g\defeq hwh^{-1}t^{-1}$ lies in $K$ with support $A_{\ell,\ell+k}$ for some $\ell\in\Z$ and $k\in\Z_{\ge0}$. Consider all conjugates $hwh^{-1}$ of $w$ such that the number $k$ is minimal. Up to replacing $h$ by $t^\ell h$, we may assume that $g=hwh^{-1}t^{-1}\in A_k$. 
	
	If $k=0$, then $hwh^{-1}=at$ for some $a\in A_0=A$. So it suffices to consider the case $k\ge1$.
	In such cases, $A_k$ is the amalgamated free product of $U\defeq A_{k-1}=A_{0,k-1}$ and $V\defeq t^{-1}A_{k-1}t=A_{1,k}$ over $W$, where $W =A_{1,k-1}$ when $k\ge2$ or $W=i_P(C)$ when $k=1$.
	Hence $g=hwh^{-1}t^{-1}$ can be written as a reduced word in $U\star_W V$, which has length at least two as $A_k$ is the support of $g$.
	
	Among all conjugates $hwh^{-1}$ of $w$ with the property that $g=hwh^{-1}t^{-1}\in A_k$ for the minimal number $k$ above, choose one such that the reduced word representing $g$ is the shortest.
	There are two cases:
	\begin{enumerate}
		\item If the reduced word representing $g$ starts with some element in $U\setminus W$, i.e. $g=u_1v_1\cdots u_m v_m$, where $m\ge 1$, $u_i\in U\setminus W$ and $v_i\in V\setminus W$ for all $i$ except that possibly $v_m=id$, in which case we must have $m\ge2$.
		
		When $v_m\in V\setminus W$, simply let $a_i=u_i$ and $b_i=tv_i t^{-1}$ and $x=id$ for $1\le i\le m$, which gives rise to the desired expression of $hwh^{-1}=gt$. In the case $v_m=id$, define $a_i,b_i$ in the same way for $i\le m-1$ and let $x=u_m$.
		\label{item: case 1 of trick pf}
		
		\item If the reduced word representing $g$ starts with some element in $V\setminus W$, i.e. 
		$g=v_0 u_1\cdots v_m$, where $m\ge 1$, $u_i\in U\setminus W$ and $v_i\in V\setminus W$ for all $i$ except that possibly $v_m=id$.
		
		If $v_m=id$, then $hwh^{-1}=gt$ is conjugate to $u_1 v_1\cdots u_m (tv_0 t^{-1}) t$. Note that  $tv_0t^{-1}\in tVt^{-1}=A_{k-1}=U$, so $g'\defeq u_1 v_1\cdots v_{m-1} [u_m (tv_0 t^{-1})]$ is a (not necessarily reduced) word in $A_k$ of length strictly less than that of the reduced word $g$, contradicting our choice of $g=hwh^{-1}t^{-1}$.
		
		If $v_m\in V\setminus W$, then $hwh^{-1}=gt$ is conjugate to $u_1 v_1\cdots u_m v_m (tv_0 t^{-1}) t$. We have $u_{m+1}\defeq tv_0 t^{-1}\in U$ as noted above, so $g'\defeq u_1 v_1\cdots u_m v_m u_{m+1}$ must be a reduced word in $A_k$ since otherwise its word length in reduced form is strictly smaller than that of $g$. Hence we can proceed as in case (\ref{item: case 1 of trick pf}).
	\end{enumerate}
\end{proof}

To prove Theorem \ref{thm: HNN general word}, we need to check that the reduced word in the new HNN extension structure as in Lemma \ref{lemma: trick} satisfies the conditions in Theorem \ref{thm: HNN main}. This easily follows from Proposition \ref{prop: inheritance} by induction.

\begin{lemma}\label{lemma: n-RF for chain of amalgam}
	For an amalgam $G_{k+1}=H_1\star_{C_1} H_2 \star_{C_2}\cdots H_k\star_{C_k} H_{k+1}$ with $k\ge 1$, suppose for some $2\le n\le \infty$ the image of each $C_i$ in $H_i$ (resp. $H_{i+1}$) is $n$-RF. 
	Then the subgroups $G_k=H_1\star_{C_1} H_2 \star_{C_2}\cdots H_k$ and $H_{k+1}$ are both $n$-RF in $G_{k+1}$.
\end{lemma}
\begin{proof}
	We proceed by induction on $k$. The base case $k=1$ is for an amalgam of $k+1=2$ free factors, and we know both free factors are $n$-RF under our assumption by Proposition \ref{prop: inheritance}. 
	Suppose the result holds for amalgams with no more than $k$ free factors and we show $G_k$ and $H_{k+1}$ are $n$-RF in $G_{k+1}$.
	Note that $G_{k+1}$ is the amalgam of $G_k$ with $H_{k+1}$ over $C_k$, so the result follows from Proposition \ref{prop: inheritance} once we show $C_k$ is $n$-RF in $G_k$.
	By the induction hypothesis $H_k$ is $n$-RF in $G_k$, and we know by assumption $C_k$ is $n$-RF in $H_k$. 
	Applying Lemma \ref{lemma: basic n-RF} (\ref{item: subgroup chain}) to the subgroup chain $C_k\le H_k\le G_k$, we see that $C_k$ is $n$-RF in $G_k$ as desired, which completes the proof.
\end{proof}

\begin{corollary}\label{cor: A_k n-RF}
	In the notation above for an HNN extension $G=A\star_C$, if for some $2\le n\le \infty$ both $(A,i_P(C))$ and $(A,i_N(C))$ are $n$-RF, 
	then the pairs $(A_{k-1},t^{-1}A_{k-2}t)$ and $(A_{k-1},A_{k-2})$ are also $n$-RF for all $k\ge1$.
\end{corollary}
\begin{proof}
	There is nothing to prove for the case $k=1$ as the assertion agrees with the assumption.
	For $k\ge2$, as we mentioned earlier, $A_{k-1}$ is the amalgam of $k+1$ copies of $A$ over $k$ copies of $C$, where $t^{-1}A_{k-2}t$ and $A_{k-2}$ are identified with the subgroups given by the amalgam of the first and last $k$ copies of $A$. Hence the assertion follows from Lemma \ref{lemma: n-RF for chain of amalgam} (and symmetry).
\end{proof}

Now we prove Theorem \ref{thm: HNN general word} (i.e. Theorem \ref{thmA: main}).

\begin{proof}[Proof of Theorem \ref{thm: HNN general word}]
	Up to replacing $w$ by $w^{-1}$ we may assume $p(w)=1$.
	By our assumption, $w$ is not conjugate to $at$ for any $a\in A$, thus by Lemma \ref{lemma: trick} $w$ is conjugate to a reduced word $w'=a_1 t^{-1} b_1 t \cdots a_m t^{-1} b_m t x t$ in the HNN extension of $A_{k-1}$ over the subgroups $A_{k-2}$ and $t^{-1}A_{k-2}t$ for some $k\in\Z_+$ and $m\ge1$. It is guaranteed by Lemma \ref{lemma: trick} that each $a_i\in A_{k-1}\setminus t^{-1}A_{k-2}t$, so it is $n$-RF and thus also $n$-RTF rel $t^{-1}A_{k-2}t$ by Corollary \ref{cor: A_k n-RF}. Similarly each $b_i$ is $n$-RF and $n$-RTF rel $A_{k-2}$.
	Note that any $w$-admissible surface $S$ for the HNN extension $H=A\star_C$ is also $w$-admissible for the HNN extension structure of $H$ above with vertex group $A_{k-1}$ by enlarging the proper subgroup $A$ to $A_{k-1}$ in Definition \ref{def: admissible}; see Remark \ref{rmk: def}. Moreover, the notion of boundary-incompressiblity stays the same in the process as well as the quantities $\deg(S)$ and $-\chi(S)$. Thus the result follows directly from Theorem \ref{thm: HNN main}.
\end{proof}
%

For a free HNN extension $H=A\star\Z$, the assumptions in Theorem \ref{thm: HNN general word} above are easy to check, so we immediately obtain Theorem \ref{thmA: less tech main} as a corollary:
\begin{corollary}\label{cor: gap for groups without small torsion}
	Let $A$ be an arbitrary group and let $p:A\star\Z\to \Z$ be the retract to $\Z$ induced by the trivial homomorphism $A\to\Z$ and $id_{\Z}$. If $w\in A\star \Z$ has $p(w)=\pm 1$, then any boundary-incompressible $w$-admissible surface $S$ has
	$$-\chi(S)\ge \frac{1}{2} \deg(S).$$
	Moreover, if each nontrivial element in $A$ has order at least $n$ for some $2\le n\le \infty$, then we have a strengthened inequality
	$$-\chi(S)\ge \left(1-\frac{1}{n}\right)\deg(S).$$
\end{corollary}
\begin{proof}
	Think of $H=A\star\Z$ as a free HNN extension. Then clearly by definition any $a\neq id\in A$ is $n$-RF rel the trivial subgroup $id$ if $a$ has order at least $n$.
	So the group-subgroup pair $(A,id)$ is $2$-RF in all cases and $n$-RF if each nontrivial element in $A$ has order at least $n$.
	Now if $w$ is not conjugate to $at^{\pm 1}$, then Theorem \ref{thm: HNN general word} implies the desired bound.
	If $w=at^{\pm 1}$, the only two turn types are $(t,id,t^{-1})$ and $(t^{-1},id,t)$, neither of which is an admissible turn (Section \ref{subsec: possible pieces}). Hence there is no boundary-incompressible $w$-admissible surface in this case, and thus the assertion is vacuous.
\end{proof}

Now we turn to applications related to the Kervaire--Laudenbach conjecture. We first deduce from Theorem \ref{thm: HNN general word} a Freiheitssatz theorem analogous to Corollary \ref{cor: injetivity for special w} for a rather general word in an HNN extension using the same argument. This should also essentially follow from \cite[Theorem 4.1]{FennRourke}.
\begin{theorem}\label{thm: HNN Freiheitssatz}
	Let $H=A\star_{C}$ be an HNN extension associated to isomorphic subgroups $C_1,C_2\le A$ so that $(A,C_1)$ and $(A,C_2)$ are both $\infty$-RF. Then for any $w\in H$ with $p(w)=\pm 1$ and not conjugate to $at^{\pm1}$, the natural map $A\to H/\llangle w\rrangle$ is injective.
\end{theorem}
\begin{proof}
	Suppose the natural map is not injective, that is, there is some $a\neq id \in A$ that lies in $\llangle w \rrangle$. As in Example \ref{example: adm surf from relations} (with $H=A\star_{C}$), this gives rise to an equation (\ref{eqn: relation}), which provides a $w$-admissible surface $S$ of degree $\deg(S)=k$ with $-\chi(S)=k-1$ (as it is a sphere with $k+1$ disks removed) for some $k\in\Z_+$. Moreover, as explained in Example \ref{example: adm surf from relations}, when $k$ is minimal among all equations of this form, $S$ is boundary-incompressible. Hence by Theorem \ref{thm: HNN general word} with $n=\infty$, we have $k-1=-\chi(S)\ge\deg(S)=k$, which leads to a contradiction. Thus the natural map must be injective.
\end{proof}

As an example, this applies to splittings of surface groups over $\Z$ as HNN extensions.
\begin{corollary}[{Howie--Saeed \cite[Theorem 1.2]{HowieSaeed_surfgrpquotient}}]\label{cor: surface group}
	For a closed orientable surface $S$, let $\beta$ be a simple non-separating loop and let $\alpha$ be a loop with algebraic intersection number $\pm1$ with $\beta$ and geometric intersection number at least $2$ with $\beta$. Then the natural map $\pi_1(S\setminus\beta)\to \pi_1(S)/\llangle w_\alpha \rrangle$ is injective, where $w_\alpha$ is the class in $\pi_1(S)$ corresponding to $\alpha$.
\end{corollary}
\begin{proof}
	Here $\pi_1(S)$ is an HNN extension of the free group $\pi_1(S\setminus\beta)$ over two $\Z$ subgroups corresponding to $\beta$. These $\Z$ subgroups are factors in the free group, so they are $\infty$-RF. The algebraic and geometric intersection numbers ensure that $w_\alpha$ meets the requirements in Theorem \ref{thm: HNN Freiheitssatz}, so the result follows.
\end{proof}
\begin{remark}\label{rmk: at is special}
	It is essential to exclude the case where $w$ is conjugate to $at^{\pm1}$ in Theorems \ref{thm: HNN general word} and  \ref{thm: HNN Freiheitssatz} and Corollary \ref{cor: surface group}.
	For instance, the surface group $H=\langle a,b,c,d\mid [a,b]=[c,d]\rangle$ is an HNN extension of the free group $A=\langle a,c,d \rangle$ over $\Z$ subgroups $C_1=\langle  a^{-1}\rangle$ and $C_2=\langle a^{-1}[c,d]\rangle$ (by identifying the chosen generators), where the generator $b$ plays the role of the stable letter $t$ in this HNN extension.
	For the word $w=b$, note that the element $[a,b]=[c,d]$ lies in the vertex group $A$ as well as the normal closure of $b$, hence it is in the kernel of $A\to H/\llangle w\rrangle$. This does not violate the results above since $b$ has geometric intersection $1$ with $a$. See \cite{Howie_error} and \cite[Exmaple in Section 2]{HowieSaeed_surfgrpquotient} for more examples. Such cases need to be treated separately.
\end{remark}

Specializing to the free HNN extension, we deduce Theorem \ref{thmA: Klyachko}, originally proved by Klyachko.
\begin{theorem}[Klyachko \cite{Klyachko}]\label{thm: Klyachko}
	Let $A$ be a torsion-free group and let $p:A\star\Z\to \Z$ be the retract to $\Z$ induced by the trivial homomorphism $A\to\Z$ and $id_{\Z}$. If $w\in A\star \Z$ has $p(w)=\pm 1$, then the natural map $A\to (A\star \Z)/\llangle w\rrangle$ is injective.
\end{theorem}
\begin{proof}
    Taking $H=A\star\Z$ as the free HNN extension of $A$ over the trivial subgroup, Theorem \ref{thm: HNN Freiheitssatz} implies the desired result except for the case where $w$ is conjugate to $at^{\pm 1}$, but the injectivity is obvious in this exceptional case.
\end{proof}

Now we prove Theorem \ref{thmA: proper power}, where the relator is a proper power but the free factor $A$ is arbitrary.
\begin{theorem}[Klyachko--Lurye \cite{KlyachkoLurye}]\label{thm: proper power}
    Let $A$ be an arbitrary group and let $p:A\star\Z\to \Z$ be the retract to $\Z$ induced by the trivial homomorphism $A\to\Z$ and $id_{\Z}$. If $w\in A\star \Z$ has $p(w)=\pm 1$, then the natural map $A\to (A\star \Z)/\llangle w^m\rrangle$ is injective for any $m\ge 2$.
\end{theorem}
\begin{proof} 
    The proof is similar to that of Theorem \ref{thm: HNN Freiheitssatz}. If the map is not injective, we obtain a surface $S$ that is a sphere with $k+1$ disks removed, where $k$ boundary components each represents $w^{\pm m}$ and the remaining one represents some $a\neq id\in A$ for some $k\ge1$. We consider $S$ as a $w$-admissible surface, then it has $\deg(S)=km$. Moreover, $S$ is boundary-incompressible when we take such an equation of $a$ with $k$ minimal. Hence by Corollary \ref{cor: gap for groups without small torsion}, we have 
    $$k-1=-\chi(S)\ge \frac{1}{2}\deg(S)=\frac{mk}{2}\ge k$$
    as $m\ge2$, and we get a contradiction. Thus the map must be injective.
\end{proof}

Note that for quotients by high powers, the problem is easier when the exponent $m$ gets larger. For instance, the case $m\ge7$ follows from small-cancellation theory \cite[Corollary 9.4]{LyndonSchupp}, and the case $m\ge6$ was proved earlier by Gonzales-Acu\~na and Short \cite{AcunaShort}. To the best knowledge of the author, the strongest previous result regarding Conjecture \ref{conj: Freiheitssatz for proper power} is a theorem of Howie \cite[Theorem A]{Howie_properpower}, which proves the case when $m\ge 4$. 

Finally we prove Theorem \ref{thmA: malnormal HNN}, generalizing Theorem \ref{thm: proper power} to HNN extensions, similar to the generalization (Theorem \ref{thm: HNN Freiheitssatz}) of Klyachko's Theorem \ref{thm: Klyachko}.
\begin{theorem}\label{thm: proper power HNN}
	Let $H=A\star_{C}$ be an HNN extension associated to isomorphic malnormal subgroups $C_1,C_2\le A$. Then for any $w\in H$ with $p(w)=\pm 1$ and not conjugate to $at^{\pm1}$, the natural map $A\to H/\llangle w^m\rrangle$ is injective for any $m\ge2$.
\end{theorem}
\begin{proof}
	By Lemma \ref{lemma: 2-RF}, the assumption that $C_1$ and $C_2$ are malnormal is equivalent to that $(A,C_1)$ and $(A,C_2)$ are $2$-RF. The rest of the proof is identical to that of Theorem \ref{thm: proper power}, except that we apply Theorem \ref{thm: HNN general word} with $n=2$ instead to bound $-\chi(S)$ by $\frac{1}{2}\deg(S)$.
\end{proof}

\begin{remark}\label{rmk: strengthening}
	In Theorem \ref{thm: HNN general word}, instead of assuming the group-subgroup pairs $(A,i_P(C))$ and $(A,i_N(C))$ to be $n$-RF, it suffices to assume that in a cyclically reduced expression of $w$ each cyclic subword $tat^{-1}$ (resp. $t^{-1}at$) has $a\in A\setminus i_P(C)$ (resp. $a\in A\setminus i_N(c)$) being $n$-RF rel $i_P(C)$ (resp. $i_N(C)$). Similarly, in Theorem \ref{thm: proper power HNN}, we can weaken the malnormality assumption to that $aC_i a^{-1}\cap C_i=\{id\}$ for letters $a$ appearing in a cyclically reduced expression of $w$, $i=1,2$.
	
	We briefly sketch how we should modify the proof to obtain this strengthening of Theorem \ref{thm: HNN general word}, which implies the strengthened version of Theorem \ref{thm: proper power HNN}. First, the proof of Lemma \ref{lemma: trick} actually can be used to show that a \emph{cyclic} conjugate of the word $w$ has the desired standard form $a_1 t^{-1} b_1 t \cdots a_m t^{-1} b_m t x t$, where each $a_i\in A_{k-1}\setminus t^{-1}A_{k-2}t$ (resp. $b_i\in A_{k-1}\setminus A_{k-2}$) is a product of elements $t^{-j}at^j$ for some $0\le j\le k-1$ and $a\in A$ that appears in the reduced word expression of $w$ so that at least one such element in the product has $j=0$ (resp. $j=k-1$) and $a\in A\setminus i_P(C)$ $n$-RF rel $i_P(C)$ (resp. $a\in A\setminus i_N(C)$ $n$-RF rel $i_N(C)$). Now an analog of Lemma \ref{lemma: n-RF for chain of amalgam} and Corollary \ref{cor: A_k n-RF} proved using Lemma \ref{lemma: inheritance, tech version} instead of Proposition \ref{prop: inheritance} implies that each $a_i$ (resp. $b_i$) is $n$-RF rel $t^{-1}A_{k-2}t$ (resp. $A_{k-2}$). So the result follows by applying Theorem \ref{thm: HNN main}.
\end{remark}

\subsection{Relative hyperbolicity and linear isoperimetric inequality}
Now we give a different application of Corollary \ref{cor: gap for groups without small torsion} to prove a linear isoperimetric inequality to deduce relative hyperbolicity of groups of the form $(A\star\Z)/\llangle w^k \rrangle$, recovering the main theorem in \cite{KlyachkoLurye}.

Fix an integer $k\ge2$. Let $A$ be an arbitrary group and let $w$ an element of $A\star\Z$ with $p(w)=\pm1$ as in Theorem \ref{thm: proper power}. 
Let $t$ be a generator of $\Z$. Then the group $G\defeq(A\star\Z)/\llangle w^k \rrangle$ has finite relative presentation $G=\langle A, t\mid w^k\rangle$. 

Recall that $G$ is hyperbolic relative to $A$ in the sense of Osin \cite[Definition 1.6]{Osin} if there is a linear isoperimetric inequality: There is some constant $C>0$ such that for any $n\in\Z_+$, any word in the alphabet $A\cup\{t,t^{-1}\}$ of length at most $n$ representing $id_G$ is a product of $m$ conjugates of $w^k$ or $w^{-k}$ for some $m\le Cn$.

The theorem below recovers the second part of the main theorem in \cite{KlyachkoLurye} with a more explicit bound on the linear isoperimetric constant. Note that this is a generalization of the classical fact that one-relator groups with torsion (i.e. those of the form $F_n/\llangle w^k\rrangle$ for $k\ge2$) are Gromov-hyperbolic. This can be seen by taking $A=F_{n-1}$ and using the fact that a group is hyperbolic if it is hyperbolic relative to a hyperbolic subgroup \cite[Corollary 2.41]{Osin}.
\begin{theorem}\label{thm: linear isoperimetric}
	In the setup above, if $k\ge2$ and $A$ has no $2$-torsion, or if $k\ge3$, then the linear isoperimetric inequality above holds for $G=(A\star\Z)/\llangle w^k\rrangle$ with the linear isoperimetric constant $C=7|w|$, where $|w|$ is the cyclically reduced word length of $w$ in the free product $A\star\Z$.
	In particular, $G$ is hyperbolic relative to $A$.
\end{theorem}
\begin{proof}
	Suppose a word $u$ in $A\cup\{t,t^{-1}\}$ of length $|u|$ represents $id_G$, then $u=w_1\cdots w_m$ in $A\star\Z$ where each $w_i$ is a conjugate of $w^k$ or $w^{-k}$. Suppose $m$ is minimal in all such equations (fixing $u$).
	Conjugating both sides of the equation, it suffices to show $m\le C|u|$ assuming $u$ is a cyclically reduced word in $A\star\Z$, as otherwise $|u|$ is even larger.
	So we may express $u$ as $u=a_1 t^{k_1}\cdots a_\ell t^{k_\ell}$, where $a_i\neq id \in A$ and $k_i\neq0\in\Z$, and $|u|=\ell+\sum |k_i|$.
	
	Note that $p(u)=\sum k_i$ (with a suitable choice of $t$) and $p(w_i)=p(w^{\pm k})=\pm k$, thus the equation $u=w_1\cdots w_m$ implies that $k$ must divide $\sum k_i$. Let $q=\sum k_i/k\in\Z$.
	
	First consider the (easier) case where $q=0$, i.e. $p(u)=\sum k_i=0$. Then $u$ is a product of conjugates of $a_1,\cdots, a_\ell$ by powers of $t$ (with no $t$ left). Hence the equation $u=w_1\cdots w_m$ can be represented by a $w$-admissible surface $S$ that is the sphere with $\ell+m$ boundary components, $\ell$ of which represent the conjugacy classes of $a_i$'s and the others represent $w^{\pm k}$.
	It is boundary-incompressible by minimality of $m$. Thus by Corollary \ref{cor: gap for groups without small torsion} we have
	$$\ell+m-2=-\chi(S)\ge \lambda\deg(S)=\lambda mk,$$
	where $\lambda\ge1/2$, and in addition $\lambda\ge 2/3$ if $A$ has no $2$-torsion.
	
	This implies
	$$|u|=\ell+\sum|k_i|> \ell-2\ge (\lambda k-1)m.$$
	Note that $\lambda k-1\ge 1/2$ if $k\ge3$ and $\lambda k-1\ge 1/3$ if $k\ge2$ and $A$ has no $2$-torsion.
	So $C_0\defeq \frac{1}{\lambda k-1}\le 3$ under either assumptions, and we have
	$m\le C_0|u|$ as desired.
	
	For the general case, multiply $w^{-kq}$ to both sides of the equation $u=w_1\cdots w_m$ to obtain a new equation $u'=w_1\cdots w_m (w^{-k\cdot \mathrm{sign}(q)})^{|q|}$.
	Then the new word $u'$ has $p(u')=0$ and its length satisfies
	$$|u'|\le |u|+\left|\sum k_i\right| |w|\le \ell+(1+|w|)\sum|k_i|.$$
	If $m'$ is the minimal number of conjugates of $w^{\pm k}$ with their product equal to $u'$, then
	$u$ is the product of at most $m'+q$ conjugates of $w^{\pm k}$, and thus $m\le m'+q$. 
	On the other hand, applying the linear isoperimetric inequality proved above to $u'$, we have
	$$m'\le C_0|u'|\le C_0 [\ell+(1+|w|)\sum|k_i|]\le (2C_0|w|)(\ell+\sum|k_i|)=2C_0|w||u|.$$
	Combining with $q\le |\sum k_i|\le \sum |k_i|\le |u|\le |w||u|$, we get
	$$m\le m'+q\le (2C_0+1)|w||u|,$$
	which completes the proof noting that $2C_0+1\le 7$ as $C_0\le 3$.
	
\end{proof}

\section{Questions}\label{sec: questions}
We conclude by listing a few questions related to our results.

\begin{question}
	Does Theorem \ref{thmA: less tech main} hold under the weaker assumption $p(w)\neq0$ (or $w$ does not conjugate into the free factor $A$), especially for the special case of $n=\infty$?
\end{question}

The algebraic trick (Lemma \ref{lemma: trick}) reducing a word with $p(w)=1$ to words of the specific form in Theorem \ref{thm: HNN main} plays a crucial role in the proof of Theorem \ref{thmA: main} (and Theorem \ref{thmA: less tech main} as its consequence). Such a trick seems unavailable for the more general setting, and thus finding a more direct proof without reducing to those words of the specific form might be a starting point to obtain such a generalization to words with $p(w)\neq0$.

The minimal complexity problem we consider here suggests the study of a quantity analogous to stable commutator length. Given a group $H$ with a proper subgroup $A$, define the \emph{geometric filling norm} of an element $w$ in $H$ relative to $A$ as 
$$\mathrm{gfill}_{H,A}(w)\defeq \inf_S \frac{-\chi(S)}{\deg(S)},$$ 
where the infimum is taken over all boundary-incompressible $w$-admissible surfaces (relative to $A$ as in Definition \ref{def: admissible}). Then our main results (Theorems \ref{thmA: less tech main} and \ref{thmA: main}) provide uniform lower bounds of $\mathrm{gfill}_{H,A}(w)$, analogous to spectral gap results of stable commutator length (e.g. \cite{CF:sclhypgrp,BBF,CH:sclgap}). 

\begin{remark}\label{rmk: comparison}
	An admissible surface $S$ relative to $A$ in the (relative) scl sense \cite[Definition 2.8]{Chen:sclBS} is also a $w$-admissible surface (relative to $A$), and geometric degree is no less than the absolute value of the algebraic degree. So it is immediate that we have $$2\scl_{H,A}(w)\ge \mathrm{gfill}_{H,A}(w);$$
	see \cite[Section 2.2]{Chen:sclBS} for the topological definition of $\scl_{H,A}$.
	
	However, our main results (Theorems \ref{thmA: less tech main} and \ref{thmA: main}) do not directly imply meaningful lower bounds of $\scl_{H,A}(w)$ because the assumption $p(w)=\pm 1$ ensures that $w$ is homologically nontrivial (relative to $A$) and $\scl_{H,A}(w)=+\infty$ by convention. Generalizing our results to cover some $w$ with $p(w)=0$ would yield lower bounds for relative scl. It could be interesting even in the simple setting of $H=A\star \Z$ with $A=\Z/2$, where it seems likely that $\scl_{H,A}(w)\ge 1/4$ for all $w$ not conjugate into $A$; compare this to Theorem \ref{thmA: less tech main} with $n=2$.
\end{remark}

Many other nice properties of stable commutator length might hold analogously for $\mathrm{gfill}_{H,A}$. For simplicity, consider below the case where $A=\{id\}$ and denote $\mathrm{gfill}_H\defeq \mathrm{gfill}_{H,\{id\}}$.

In comparison with rationality results of stable commutator length (e.g. \cite{Cal:rational,Chen:sclBS}), which are related to finding surface subgroups, we ask:
\begin{question}
	For a free group $H$, is $\mathrm{gfill}_H(w)$ rational for each $w$? Is there a $w$-admissible surface realizing the infimum in the definition of $\mathrm{gfill}_H(w)$ for each $w$? If so, is there an algorithm to find such a minimizer? What about other groups?
\end{question}

Another fascinating part of stable commutator length is the Bavard duality \cite{Bavard_duality} relating it to homogeneous quasimorphisms. This is an important tool to obtain lower bounds of stable commutator length and prove spectral gap results. Hence it is natural to ask:
\begin{question}
	Is there an analog of Bavard's duality for $\mathrm{gfill}_H$? What are the dual objects?
\end{question}


\bibliographystyle{alpha}
\bibliography{kervaire}

\end{document}